\documentclass[a4paper,10pt,reqno, english]{amsart} 
\usepackage{amssymb}
\usepackage{graphicx}
\usepackage{amsmath,amsthm}
\usepackage{amsfonts,amssymb,enumerate}
\usepackage{url,paralist}
\usepackage{mathtools}
\usepackage[arrow,curve,matrix,tips,2cell]{xy}  
  \SelectTips{eu}{10} \UseTips
  \UseAllTwocells

\usepackage{mathtools}  
\usepackage[colorlinks=true,urlcolor=blue,linkcolor=red,citecolor=magenta]{hyperref}
\usepackage{tikz}
\usepackage{pdfsync}
\usepackage{hyperref}
\usepackage{color}
\usepackage{enumerate,amssymb}

\theoremstyle{plain}
\newtheorem{theorem}{Theorem}[section]

\newtheorem{lemma}[theorem]{Lemma}

\newtheorem{corollary}[theorem]{Corollary}

\newtheorem*{theorem*}{Theorem}
\newtheorem{question}[theorem]{Question}

\theoremstyle{definition}

\newcommand{\BIGOP}[1]{\mathop{\mathchoice%
{\raise-0.22em\hbox{\huge $#1$}}%
{\raise-0.05em\hbox{\Large $#1$}}{\hbox{\large $#1$}}{#1}}}

\newcommand\RP{\mathbb{R}{\rm P}}
\newcommand\CP{\mathbb{C}{\rm P}}
\newcommand{\R}{\mathbb{R}}

\newcommand{\C}{\mathbb{C}}

\newcommand{\F}{\mathbb{F}}
\newcommand{\B}{\mathrm{B}}
\newcommand{\E}{\mathrm{E}}

\newcommand\Sym{\mathfrak S}

\newcommand{\im}{\operatorname{im}}

\newcommand{\Id}{\operatorname{Id}}

\newcommand{\rank}{\operatorname{rank}}

\newcommand{\diam}{\operatorname{diam}}

\newcommand{\conf}{\operatorname{Conf}}
\newcommand{\norm}[1]{\left\lVert#1\right\rVert}
\newcommand{\length}{\operatorname{length}}
\newcommand{\interior}{\operatorname{int}}
\newcommand{\spann}{\operatorname{span}}

\newcounter{commentcounter}

\begin{document}

\title{Local multiplicity of continuous maps between manifolds}

\author[Blagojevi\'c]{Pavle V. M. Blagojevi\'{c}} 
\thanks{The research by Pavle V. M. Blagojevi\'{c} leading to these results has
        received funding from DFG via Berlin Mathematical School, and the grant ON 174008 of the Serbian Ministry of Education and Science.}
\address{Inst. Math., FU Berlin, Arnimallee 2, 14195 Berlin, Germany\hfill\break%
\mbox{\hspace{4mm}}Mat. Institut SANU, Knez Mihailova 36, 11001 Beograd, Serbia}
\email{blagojevic@math.fu-berlin.de} 
\author[Roman Karasev]{Roman Karasev}
\thanks{The research by Roman Karasev leading to these results has
        received funding from the Russian Foundation for Basic Research grants 15-31-20403 (mol\_a\_ved) and 15-01-99563 A}
\address{Moscow Institute of Physics and Technology, Institutskiy per. 9, Dolgoprudny, \hfill\break
\mbox{\hspace{4mm}}Russia 141700\hfill\break
\mbox{\hspace{4mm}}Institute for Information Transmission Problems RAS, Bolshoy Karetny per. 19, \hfill\break
\mbox{\hspace{4mm}}Moscow, Russia 127994}
\email{r\_n\_karasev@mail.ru}


\date{\today}



\date{}

\maketitle

\begin{abstract}
Let $M$ and $N$ be smooth (real or complex) manifolds, and let $M$ be equipped with some Riemannian metric.
A continuous map $f\colon M\longrightarrow N$ admits a local $k$-multiplicity if, for every real number $\omega >0$, there exist $k$ pairwise distinct points $x_1,\ldots,x_k$ in $M$ such that $f(x_1)=\cdots=f(x_k)$ and $\diam\{x_1,\ldots,x_k\}<\omega$.
In this paper we systematically study the existence of local $k$-mutiplicities and derive criteria for the existence of local $k$-multiplicity in terms of Stiefel--Whitney classes and Chern classes of the vector bundle $f^*\tau N\oplus(-\tau M)$.
For example, as a corollary of one criterion we deduce that for $k\geq 2$ a power of $2$, $M$ a compact smooth manifold with the integer $s:=\max\{\ell : \bar{w}_{\ell}(M)\neq 0\}$, and $N$ a parallelizable smooth manifold, if $s\geq \dim N-\dim M+1$ and $\bar{w}_{s}(M)^{k-1}\neq 0$, any continuous map $M\longrightarrow N$ admits a local $k$-multiplicity.
Furthermore, as a special case of this corollary we recover, when $k=2$, the classical criterion for the non-existence of an immersion $M\looparrowright N$ between manifolds $M$ and $N$.
	
\end{abstract}


\section{Introduction}

Let $M$ and $N$ be smooth manifolds, and let $k\geq 2$ be an integer. 
A continuous map $f\colon M\longrightarrow N$ {\em admits a $k$-multiplicity} if there exist $k$ pairwise distinct points $x_1,\ldots,x_k$ on $M$ such that 
\[
f(x_1)=\cdots=f(x_k).
\]
For example, a continuous map $f$ that admits a $2$-multiplicity is not a (topological) embedding.
An interesting result of Gromov on $k>2$ multiplicities \cite[p.\,447]{Gromov2010} shows that for every $m$-dimensional manifold $M$ there exists a smooth map $M\longrightarrow\R^m$ that does not admit $k$-multiplicity for $k\geq 4m+1$.

\medskip
Let us in addition assume that the manifold $M$ is equipped with some Riemannian metric.
A continuous map $f\colon M\longrightarrow N$ {\em admits a local $k$-multiplicity} if, for every real number $\omega >0$,  there exist $k$ pairwise distinct points $x_1,\ldots,x_k$ in $M$ such that 
\[
f(x_1)=\cdots=f(x_k)
\qquad\text{and}\qquad
\diam\{x_1,\ldots,x_k\}<\omega.
\]
For example, the map $f\colon \C\longrightarrow\C$ given by $f(z)=z^k$ admits a local $k$-multiplicity. 
Existence of a local $k$-multiplicity for the map $f$ implies the existence of a $k$-multiplicity for the same map. 
In the case $k=2$ if a smooth map $f\colon M\longrightarrow N$ admits a local $2$-multiplicity, then $f$ is not an immersion.
The property of being an immersion is of course stronger than the property of having no local $2$-multiplicity.

Although the notion of the local multiplicity for continuous maps is a natural extension of the non-immersibility property for smooth maps, it has not been not studied systematically before.
In this paper we develop, and apply, a topological framework to study the existence of local multiplicity of continuous maps between real or complex manifolds.

\subsection{The statements of the main results}

The central results of this paper are the following two theorems that, for a given continuous map, give a cohomological criterion for the existence of local multiplicity.  

\begin{theorem}
	\label{th : main 0}
	Let $k\geq 2$ be a power of $2$, let $M$ be a compact smooth manifold, let $N$ be a smooth manifold, and let $f\colon M \longrightarrow N$ be a continuous map.
	Denote by $w_i:=w_i(f^*\tau N\oplus(-\tau M))$ the $i$-th Stiefel--Whitney class of the vector bundle $f^*\tau N\oplus(-\tau M)$ for $i\geq 0$, and $w_i=0$ for $i<0$.
	
	If there exists an integer $s\ge 0$ such that the characteristic class 
	\[
	u_s(f^*\tau N\oplus(-\tau M)):=\det (w_{\dim N-\dim M + 1 + s -i+j})_{1\leq i,j\leq k-1}
	\]
	does not vanish, then the continuous map $f$ admits a local $k$-multiplicity.
\end{theorem}

\medskip
In the case when both manifolds $M$ and $N$ allow almost complex structure an additional criterion can be used.

\begin{theorem}
	\label{th : main 0-C}
	Let $k\geq 2$ be an odd prime, let $M$ be a compact smooth almost complex manifold, and let $f\colon M \longrightarrow N$ be a continuous map.
	Furthermore, let us denote by $c_i:=c_i(f^*\tau N\oplus(-\tau M))$ the $i$-th Chern class mod $k$ of the complex vector bundle $f^*\tau N\oplus(-\tau M)$ for $i\geq 0$, and $c_i=0$ for $i<0$.
	If there exists an integer $s\ge 0$ such that the characteristic class 
	\[
	v_s(f^*\tau N\oplus(-\tau M)):=\det (c_{\dim N - \dim M + 1 + s -i+j})_{1\leq i,j\leq k-1}
	\]
	does not vanish, then the continuous map $f$ admits a local $k$-multiplicity.
\end{theorem}

\noindent
In both theorems $-\tau M$ denotes the inverse (real or complex) vector bundle of the tangent vector bundle $\tau M$.

\medskip
A special case of Theorem \ref{th : main 0} is the following result. 

\begin{theorem}
	\label{th : main 1}
	Let $k\geq 2$ be a power of $2$, let $M$ be a compact smooth manifold with $s:=\max\{\ell : \bar{w}_{\ell}(M)\neq 0\}$, and let $N$ be a parallelizable smooth manifold.
	If $s\ge \dim N - \dim M + 1$ and $\bar{w}_{s}(M)^{k-1}\neq 0$, then any continuous map $M\longrightarrow N$ admits a local $k$-multiplicity.
\end{theorem}

\noindent
Here $\bar{w}_{i}(M)$ denotes the dual $i$-th Stiefel--Whitney class of the tangent vector bundle $\tau M$.
An assumption that the manifold $N$ is parallelizable in Theorem~\ref{th : main 1} can be weakened and we can assume that $w(f^*\tau N)=1$ instead.
Moreover, in the case when $k=2$ the statement of Theorem~\ref{th : main 1} yields the classical obstruction for the non-existence of an immersion $M\looparrowright N$ between manifolds $M$ and $N$, see for example \cite[Cor.\,3.5]{Milnor1974} 

\medskip
A similar consequence of Theorem \ref{th : main 0-C} can be derived in the case, when $k$ is an odd prime.

\begin{theorem}
	\label{th : main 1-C}
	Let $k\geq 2$ be an odd prime, let $M$ be a compact smooth almost complex manifold with $s:=\max\{\ell : \bar{c}_{\ell}(M)\neq 0\}$, let $N$ be a parallelizable smooth complex manifold, and let $f\colon M \longrightarrow N$ be a continuous map.
	If $s\ge \dim N - \dim M + 1$ and $\bar{c}_{s}(M)^{k-1}\neq 0$, then any continuous map $M\longrightarrow N$ admits a local $k$-multiplicity.
\end{theorem}

\noindent
Here $\bar{c}_{i}(M)$ denotes the $i$-th Chern class mod $k$ of the inverse of the complex tangent vector bundle $\tau M$, and {\bf not} the $i$-th Chern class of the dual complex vector bundle.

\medskip
Using well known facts about the Stiefel--Whitney classes of tangent bundles of projective spaces we derive following corollaries of Theorem \ref{th : main 1}.

\begin{corollary}
	\label{cor : main 2}
	Let $a\geq 1$ and $\ell\geq1$ be integers, let $k\geq 2$ be a power of $2$, and let $k(a+1)\leq 2^{\ell}-1$.
	Then any continuous map  
	\[
	\RP^{2^{\ell}-2-a}\longrightarrow \R^{2^{\ell}-2}
	\]
	admits a local $k$-multiplicity.
\end{corollary}

\noindent
For $k=2$ this corollary recovers, and slightly extends, the result of Milnor \cite{Milnor1957} from 1957 and implies that there is no immersion $\RP^{2^{\ell-1}}\longrightarrow \R^{2^{\ell}-2}$, or more precisely that there cannot exist a continuous map $\RP^{2^{\ell-1}}\longrightarrow \R^{2^{\ell}-2}$ that does not admit a local $2$-multiplicity.

\medskip
The next corollary is a consequence of Theorem~\ref{th : main 0} rather then Theorem~\ref{th : main 1} even the proof is almost identical to the proof of Corollary~\ref{cor : main 2}.

\begin{corollary}
	\label{cor : main 2.5}
	Let $a\geq 1$ and $\ell\geq1$ be integers, let $k\geq 2$ be a power of $2$, and let $k(a+1)\leq 2^{\ell}-1$.
	Then any continuous map  
	\[
	\RP^{2^{\ell}-2-a}\longrightarrow S^{2^{\ell}-2}
	\]
	admits a local $k$-multiplicity.
\end{corollary}

The next consequence is obtained via direct application of Theorem~\ref{th : main 1}.

\begin{corollary}
	\label{cor : main 3}
	Let $a\geq 1$ and $\ell\geq1$ be integers, let $k\geq 2$ be a power of $2$, and let $k(a-1)\leq 2^{\ell}-1$.
	Then any continuous map  
	\[
	\CP^{2^{\ell}-a}\longrightarrow \R^{2^{\ell+1}-3}
	\]
	admits a local $k$-multiplicity.
\end{corollary} 

\medskip
Using the knowledge on Chern classes of tangent bundles of the complex projective spaces we get the following corollary of Theorem \ref{th : main 1-C}.

\begin{corollary}
	\label{cor : main 4}
	Let $a\geq 1$ and $\ell\geq1$ be integers, let $k$ be an odd prime, and let $2\leq a\leq\tfrac{k^{\ell}+1}{2}$.
	If $k(a-1)\leq k^{\ell}-1$ then any continuous map  
	\[
	\CP^{k^{\ell}-a}\longrightarrow \C^{k^{\ell}-2}
	\]
	admits a local $k$-multiplicity.
\end{corollary} 

\medskip
The previous results motivates many natural question.
We state the first and most obvious one that, in the case of $k=2$, extends the well known problem of the existence of an immersion between smooth manifolds.

\begin{question}
	Let $M$ be a smooth compact manifold of dimension $m$, and let $k\geq 2$ be an integer.
	What is the minimal dimension $n$ such that there exists a continuous map $f\colon M\longrightarrow \R^n$ such that $f$ does not admit a local $k$-multiplicity?
\end{question}

\noindent
It is our belief that in the process of answering this or similar questions about local multiplicities many new fascinating ideas will come to life as was the case when the immersion conjecture was studied by Brown and Peterson \cite{Brown1977} and resolved by Ralph Cohen \cite{CohenR1985} during the 1970s and 1980s.

\subsubsection*{Acknowledgements.}
We are grateful to Peter Landweber for, as always, very useful and precise comments and suggestions.
\subsection{Geometric applications}
The concept of $k$-multiplicity, as also the concept of local $k$-multiplicity, are natural topological properties of a continuous map.
After studying these properties from a topological point of view an immediate question arises: 
{\em Can the existence of $k$-multiplicity or local $k$-multiplicity of a continuous map be used in solving problems outside the obvious realm of topology?}
In the following, using Corollary~\ref{cor : main 2.5}, we demonstrate how to obtain lower bounds for several questions in convex geometry.

\medskip
In 1963 Branko~Gr\"{u}nbaum \cite[Sec.\,6.5]{Grunbaum1963} posed many interesting problems.
We consider the following two problems.

Let $K\subset \R^{m}$ be a convex body, that is a convex compact subset of $\R^{m}$ with non-empty interior.
An \emph{affine diameter} of a convex body $K$ {\em in direction} $\ell\in\RP^{m-1}$ is any affine line $\ell+v$ for $v\in\R^{m}$ with the property that
\[
\length (K\cap(\ell+v))=\max \{ \length (K\cap(\ell+u)) : y\in\R^{m}\}.
\]
Here $\length$ denotes the length of an interval.
In general an affine diameter in direction $\ell\in\RP^{m-1}$ is not unique.
For example, if $K$ is a square in $\R^2$, then in any direction parallel to one of the edges there are infinitely many affine dimeters.
On the other hand, if $K$ is a strictly convex body then in any direction there is a unique affine diameter.

\begin{figure}
\centering
\includegraphics[scale=0.7]{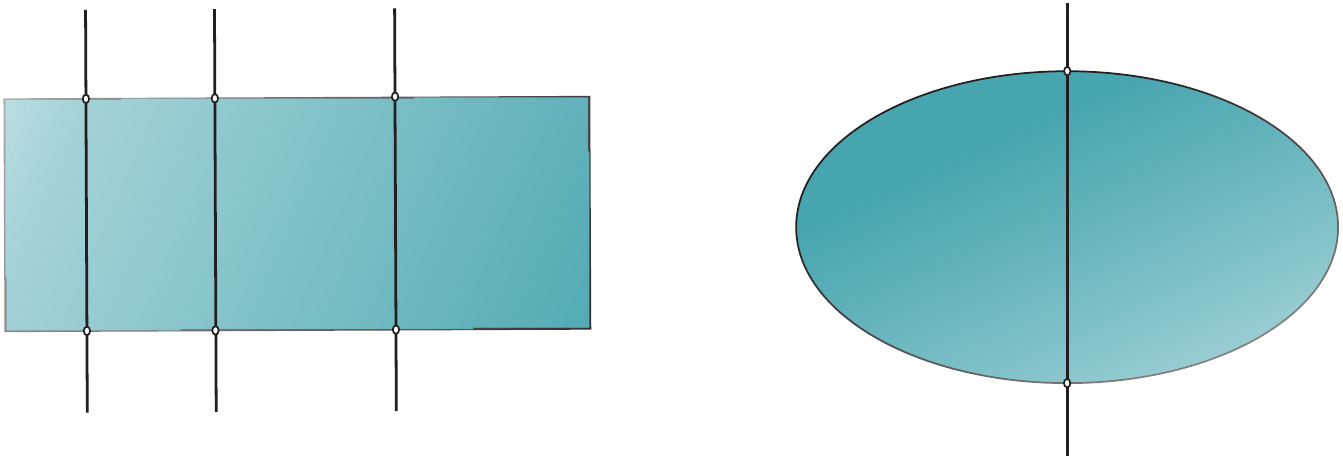}
\caption{\small Affine diameters in a given direction.}
\end{figure}

The first question of Gr\"{u}nbaum we consider asks for the number of affine diameters of a convex body intersecting in a single point.

\begin{question}
\label{aff-diam-mult-prob}
Let $K\subset \R^{m}$ be a convex body. 
Is there at least $m+1$ pairwise distinct affine diameters $\ell_1+v_1,\ldots, \ell_{m+1}+v_{m+1}$ of $K$ such that
\[
\emptyset\neq (\ell_1+v_1)\cap\cdots\cap(\ell_{m+1}+v_{m+1})\cap \interior K?
\]	
\end{question}

\noindent
This problem was studied intensively by many authors that employed diverse methods in addressing this question, see for example the work of B\'ar\'any et al., \cite{Barany1990}, \cite{Barany2016}, and for survey of known result \cite{Soltan2005}.
Here we relate the number of pairwise distinct affine diameters intersecting in a single point inside a strictly convex body with a multiplicity of a continuous map $\RP^m\longrightarrow S^m$.

\begin{theorem}
\label{aff-diam-thm}
	Let $K\subset \R^{m}$ be a strictly convex body. 
	There exists a continuous map $f_K\colon \RP^m\longrightarrow S^m$ with the property that if $f_K$ admits a $k$-multiplicity, then there exist $k$ pairwise distinct affine diameters $\ell_1+v_1,\ldots, \ell_{k}+v_{k}$ of $K$ such that
\[
\emptyset\neq (\ell_1+v_1)\cap\cdots\cap(\ell_{k}+v_{k})\cap \interior K.
\]	
\end{theorem}
\begin{proof}
The space of all affine lines in $\R^{m}$ can be identified with the total space $E(\gamma^{m-1}_m)$ of the tautological bundle $\gamma^{m-1}_m$ over the Grassmann manifold $G_{m-1}(\R^m)$ of all $(m-1)$-dimensional vector subspaces of $\R^{m}$.
Indeed, let $w\in S^{m-1}$ be a unit vector and let $\ell_{w}=\spann\{w\}\in\RP^{m-1}$ be the corresponding $1$-dimensional vector subspace of $\R^{m}$.  
Let $v\in \R^m$. 
Then the correspondence between affine lines and points in $E(\gamma^{m-1}_m)$ is given by 
\[
\ell + v \longleftrightarrow ( \ell^{\perp},\,v- (w_{\ell}\cdot v)\, w_{\ell}).
\]
Here ``$\cdot$'' denotes the standard inner product in $\R^{m}$, while $\ell^{\perp}$ the orthogonal complement of $\ell$.
Furthermore, recall that the projective space $\RP^{m-1}$ can be identified with the Grassmann manifold $G_{m-1}(\R^m)$ via the homeomorphism
\begin{equation}
	\label{eq : identification}
	\RP^{m-1}\ni\ell\longleftrightarrow\ell^{\perp}\in G_{m-1}(\R^m).
\end{equation}

The convex body $K$ is strictly convex and therefore for every $\ell\in \RP^{m-1}$ there exists a unique affine diameter in direction $\ell$.
The choice of an affine diameter of the strictly convex body $K$ defines a continuous map $s_K\colon\RP^{m-1}\longrightarrow E(\gamma^{m-1}_m)$ from the space of all directions  to the space of all affine lines in $\R^{m}$.
If the projective space $\RP^{m-1}$ is identified with the Grassmann manifold $G_{m-1}(\R^m)$ via \eqref{eq : identification}, then the function $s_K$ becomes a section of the tautological bundle $\gamma^{m-1}_m$.

Next we consider the vector bundle $\gamma_m^{m-1}\oplus (\gamma_m^{m-1})^{\perp}$ over $G_{m-1}(\R^m)$, i.e., over $\RP^{m-1}$  via the identification \eqref{eq : identification}.
The subspace of the total space 
\begin{multline*}
X:=\{ (\ell,s_K(\ell)\oplus u) \in E(\gamma_m^{m-1}\oplus (\gamma_m^{m-1})^{\perp}) : 
\\
\ell\in\RP^{m-1}, \, s_K(\ell)\in \ell^{\perp}, \, u\in \ell , \, u+s_K(\ell)\in\interior(K) \}	
\end{multline*}
is homeomorphic to the total space of a disc bundle of the vector bundle $\gamma_m^1$ over $\RP^{m-1}$.

Let $g_K\colon X \longrightarrow \interior (K)$ be a continuous map defined by 
\[
(\ell,s_K(\ell)\oplus u) \longmapsto s_K(\ell)+ u.
\]
The map $g_K$ is a proper map and thus it induces a continuous map between one point compactifications $f_K\colon\widehat{X}\longrightarrow \widehat{\interior (K)}$.
Since $\interior (K)$ is homeomorphic to an $m$-dimensional disc, $\widehat{\interior (K)}\approx S^m$.
The one point compactification $\widehat{X}$ is homeomorphic to the Thom space of the vector bundle $\gamma_m^1$, and therefore it is homeomorphic to $\RP^m$, see \cite[Prop.\,p.68]{Stong1968}. 
The map $f_K$ has the desired property: If there exist $k$ pairwise distinct points $x_1,\ldots,x_k\in\widehat{X}\approx\RP^m$ such that $f_K(x_1)=\cdots=f_K(x_k)$ then 
\[
x_1=(\ell_1,s_K(\ell_1)\oplus u_1)\in X, \ \ldots \ , x_k=(\ell_k,s_K(\ell_k)\oplus u_k)\in X,
\]
and
\[
s_K(\ell_1)+ u_1=\cdots=s_K(\ell_k)+ u_k\in (\ell_1+s_K(\ell_1))\cap\cdots\cap(\ell_{k}+s_K(\ell_k))\cap \interior K.
\] 
This concludes the proof of the theorem.
\end{proof}

\noindent
In the case when $m=2^{\ell}-2$ for $\ell\geq 2$ from Corollary \ref{cor : main 2.5} we get that: For every strictly convex body $K\subset \R^{m}$ there exist at least $k=2^{\ell-1}$ pairwise distinct affine diameters $\ell_1+v_1,\ldots, \ell_{k}+v_{k}$ of $K$ such that
\[
\emptyset\neq (\ell_1+v_1)\cap\cdots\cap(\ell_{k}+v_{k})\cap \interior K.
\]	
 
\medskip
The second question indicated by Gr\"{u}nbaum in the same publication \cite{Grunbaum1963} is the following one. 
Let $K\subset \R^{m}$ be a convex body. For every affine hyperplane $H$ let, for example, 
\begin{compactitem}[\ \ $\circ$]
	\item $m(H)$ denote the center of mass of the convex body $K\cap H$ in $H$, and let
	\item $s(H)$ denote the Steiner point of $K\cap H$ in $H$, that is the center of mass of the Gaussian curvature of $\partial K\cap H$.
\end{compactitem}
whenever $H\cap K$ is non-empty.

\begin{figure}
\centering
\includegraphics[scale=1.2]{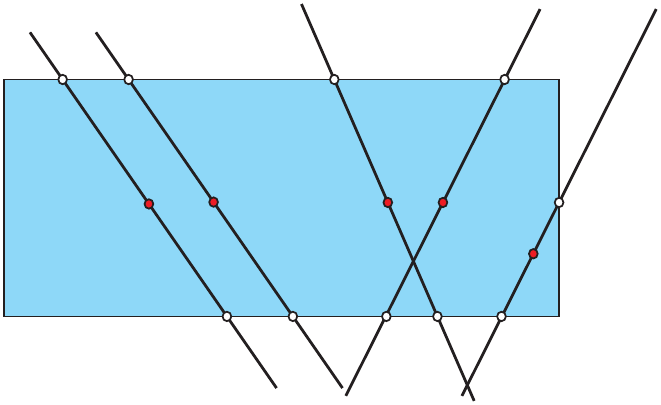}
\caption{\small Selection of mass centers of hyperplane sections.}
\end{figure}

\begin{question}
What is the maximal number $k$ such that, for any strictly convex $K$, there exist $k$ pairwise distinct affine hyperplanes $H_1,\ldots, H_k$ whose centers of mass, or Steiner points coincide, that is
\[
m(H_1)=\cdots=m(H_k),
\quad\qquad\text{or}\qquad\quad
s(H_1)=\cdots=s(H_k).
\]
\end{question}

\noindent
Relationship of a $k$-multiplicity and a coincidence of centers of mass or Steiner points is explained in the following theorem.

\begin{theorem}
	Let $K\subset \R^{m}$ be a convex body. There exists a continuous map $f_K\colon \RP^m\longrightarrow S^m$ such that if $f_K$ admits a $k$-multiplicity, then there exist $k$ pairwise distinct affine hyperplanes $H_1,\ldots, H_k$ such that
\[
m(H_1)=\cdots=m(H_k)\qquad\text{or}\qquad s(H_1)=\cdots=s(H_k).
\]	
\end{theorem}

\begin{proof}
The space of all affine hyperplanes in $\mathbb R^m$ can be identified with the total space $E(\gamma_m^1)$ of the tautological bundle $\gamma_m^1$ over $\RP^{m-1}$.
Indeed, the affine hyperplane $H$ corresponds to the point $(\ell,\ell\cap H)\in E(\gamma_m^1)$ where $\ell$ is the line through the origin perpendicular to $H$.
Let us consider the space
\[
Y := \{H \in E(\gamma_m^1) : H\cap\interior K \neq \emptyset\}.
\]
Similarly to the proof of Theorem~\ref{aff-diam-thm}, this is an open segment subbundle of $E(\gamma_m^1)$ and it can be identified with the whole $E(\gamma_m^1)$. The assignment of $m(H)$ to every $H\in Y$ determines a continuous proper map 
\[
g_K : Y \longrightarrow \interior K,
\]
which has an extension to the one-point compactifications to give
\[
f_K : \RP^m \longrightarrow S^m.
\]
The $k$-multiplicity of $f_K$ gives $k$ pairwise distinct hyperplanes with coinciding centers.
\end{proof}

\noindent
For example, Corollary \ref{cor : main 2.5} implies that for $m=2^\ell-2$ and a convex body $K\in\R^m$ there exist $k=2^{\ell-1}$ pairwise distinct affine hyperplanes whose centers of mass, or Steiner points, coincide.

\medskip
By restricting to hyperplanes passing through the origin the above argument works well to prove:

\begin{theorem}
	Let $K\subset \R^{m}$ be a convex body such that $0\in\interior K$. There exists a continuous map $f_K\colon \RP^{m-1}\longrightarrow \R^m$ such that if $f_K$ admits a $k$-multiplicity, then there exist $k$ pairwise distinct linear hyperplanes $H_1,\ldots, H_k$ such that
\[
m(H_1)=\cdots=m(H_k),
\]	
or the same with $s(H_i)$.
\end{theorem}

\noindent
An example of an explicit bound follows from Corollary \ref{cor : main 2.5} for  $m=2^\ell-2$ and   $k=2^{\ell-2}$. Then for any convex body $K\subset \R^{m}$ there exist $k$ pairwise distinct linear hyperplanes $H_1,\ldots, H_k$ whose centers of mass coincide.

\section{Multiplicity of fiberwise maps between vector bundles}

In this section we introduce and study the notion of $k$-multiplicity of a fiberwise map between vector bundles.
Then we derive a criterion, which guarantees that, for an integer $k\geq 2$ and any two vector bundles over the same base space, any continuous fiberwise map between them admits a $k$-multiplicity.

\medskip
Let $\xi$ be a vector bundle over $X$. 
Then $E(\xi)$ denotes the total space of $\xi$, $p_{\xi}\colon E(\xi)\longrightarrow X$ denotes the corresponding projection map, and $F_x(\xi)$ stands for the fiber of $\xi$ over the point $x\in X$.

\medskip
Let $X$ be a topological space with the homotopy type of a finite CW-complex, and let $k\geq 2$ be an integer.
Consider vector bundles $\xi$ and $\eta$ over $X$, and a continuous fiberwise map $\Phi\colon\xi\longrightarrow\eta$ between them.
The map $\Phi$ {\em admits a $k$-multiplicity} if there exists a point $x\in X$ in the base space and $k$ pairwise distinct vectors $v_1,\ldots,v_k$ in $F_x(\xi)$, the fiber of $\xi$ over $x$, such that
\[
\Phi(x, v_1)=\cdots=\Phi(x,v_k).
\]
Here we abuse notation and instead of writing $\Phi(e)$ where $e\in E(\xi)$ we keep track of the fiber to which $e$ belongs to.

\medskip
The {\em fiberwise configuration space} of the continuous map $p_{\xi}$ is a subspace of the configuration space  $\conf(E(\xi),k)$ defined by 
\begin{eqnarray*}
\conf_{p_{\xi}}(E(\xi),k)  & = & \{ (e_1,\ldots,e_k)\in \conf(E(\xi),k) : p_{\xi}(e_i)=p_{\xi}(e_j)\text{ for all }i,j\}	\\
				   & = & \{ (x;v_1,\ldots,v_k): x\in X\text{ and }(v_1,\ldots,v_k)\in\conf(F_x(\xi),k) \}\\
				   & = & \{ (x;v_1\oplus\cdots\oplus v_k)\in E(\xi^{\oplus k}): x\in X\text{ and }v_i\neq v_j\text{ for }i\neq j \}.
\end{eqnarray*}
Here $\conf(E(\xi),k)$ denotes the classical configuration space of $k$ ordered pairwise distinct points in $E(\xi)$.
The projection map $p_{\xi}$ of the vector bundle $\xi$ induces the following bundle 
\[
\conf_{p_{\xi}}(E(\xi),k)\longrightarrow X,\qquad (x;v_1,\ldots,v_k)\longmapsto x.
\]

\medskip
Any continuous fiberwise map $\Phi\colon\xi\longrightarrow\eta$ gives rise of the following continuous fiberwise map
\begin{equation}
	\label{diag : Morphism - 1}
	\xymatrix{
				\conf_{p_{\xi}}(E(\xi),k)\ar[rr]^-{\Phi^{\oplus k}}\ar[dr] &     & E(\eta^{\oplus k})\ar[dl]\\
                         &X&          
}
\end{equation}
defined by 
\[
(x,v_1\oplus\cdots\oplus v_k)\longmapsto \Phi(x, v_1)\oplus\cdots\oplus\Phi(x,v_k).
\]
This is a restriction of the continuous fiberwise map $\Phi^{\oplus k}\colon\xi^{\oplus k}\longrightarrow\eta^{\oplus k}$.
The fiberwise configuration space $\conf_{p_{\xi}}(E(\xi),k)$ as well as the total space $E(\eta^{\oplus k})$ are equipped with the obvious fiberwise $\Sym_k$-actions in such a way that the fiberwise map $\Phi^{\oplus k}$ becomes an $\Sym_k$-equivariant map.
The action on $\conf_{p_{\xi}}(E(\xi),k)$ permutes $k$ pairwise distinct vectors in each fiber and therefore is free. 
Let $\underline{\R^{\oplus k}}$ denote the trivial bundle over $X$ with fiber $\R^{\oplus k}$ equipped with the action of $\Sym_k$ that permutes summands. 
Then $\eta^{\oplus k}\cong \eta\otimes_{\R}\underline{\R^k}$, where the action on the tensor product is the diagonal action, assuming trivial action on $\eta$.

\medskip
The first step in the finding a useful criterion for the existence of a $k$-multiplicity of a continuous fiberwise map is the following stabilization lemma.
\begin{lemma}
	\label{lem : stabilization}
	Let $\xi$, $\zeta$ and $\eta$ be vector bundles over the space $X$.
    The continuous fiberwise map $\Phi\colon\xi\longrightarrow\eta$ admits a $k$-multiplicity if and only if the continuous fiberwise map $\Phi\oplus \Id\colon\xi\oplus\zeta\longrightarrow\eta\oplus\zeta$ also admits a $k$-multiplicity.
\end{lemma}

\noindent
Here $\Id\colon\zeta\longrightarrow\zeta$ denotes the identity map that is also a fiberwise map.
\begin{proof}
	{\bf (1)}
	Let the fiberwise map $\Phi\colon\xi\longrightarrow\eta$ admit a $k$-multiplicity.
	Consequently, there exists $x_0\in X$ and $(v_1,\ldots,v_k)\in \conf(F_{x_0}(\xi),k)$ such that 
	\[
	\Phi(x_0, v_1)=\cdots=\Phi(x_0,v_k).
	\]
	Since, $(\Phi\oplus \Id) (x,v\oplus u) = (x,\Phi(x,v)\oplus u)$ for every $(x,v\oplus u)\in F_x(\xi\oplus\eta)\cong F_x(\xi)\oplus F_x(\eta)$, we get the following $k$-multiplicity of $\Phi\oplus \Id$:
	\[
	(\Phi\oplus\Id)(x_0, v_1\oplus 0)=\cdots=(\Phi\oplus\Id)(x_0,v_k\oplus 0).
	\]
	Thus, if $\Phi$ admits a $k$-multiplicity, then $\Phi\oplus \Id$ also admits a $k$-multiplicity.
	
	\noindent
	{\bf (2)} Now, let the fiberwise map $\Phi\oplus \Id\colon\xi\oplus\zeta\longrightarrow\eta\oplus\zeta$ admit a $k$-multiplicity.
	Thus, there exists $x_0\in X$ and $(v_1\oplus u_1 ,\ldots,v_k\oplus u_k) \in \conf(F_{x_0}(\xi\oplus\zeta),k)$ such that $(\Phi\oplus\Id)(x_0, v_1\oplus u_1)=\cdots=(\Phi\oplus \Id)(x_0,v_k\oplus u_k)$, that is
	\[
	(x_0, \Phi(x_0,v_1)\oplus u_1)=\cdots=(x_0,\Phi(x_0,v_k)\oplus u_k)\in 
	F_{x_0}(\xi\oplus\zeta)\cong F_{x_0}(\xi)\oplus F_{x_0}(\zeta).
	\]
	Consequently, $u_1=\cdots=u_k$ and $\Phi(x_0,v_1)=\cdots=\Phi(x_0,v_k)$. 
	Hence, if $\Phi\oplus \Id$ admits a $k$-multiplicity, then $\Phi$ also admits a $k$-multiplicity.
\end{proof}

\medskip
Instead of studying $k$-multiplicity of $\Phi\colon\xi\longrightarrow\eta$ directly we are going to use Lemma~\ref{lem : stabilization} and consider $k$-multiplicity of $\Phi\oplus \Id\colon\xi\oplus\zeta\longrightarrow\eta\oplus\zeta$ where $\zeta$ is an inverse bundle of $\xi$, that means $\xi\oplus\zeta$ is a trivial vector bundle.
When convenient we denote the bundle $\zeta$ by $-\xi$.
In that case the fiberwise configuration space  associated to the projection map  $p_{\xi\oplus\zeta}\colon E(\xi\oplus\zeta)\longrightarrow X$ becomes a trivial bundle and decomposes as follows
\[
\conf_{p_{\xi\oplus\zeta}}(E(\xi\oplus\zeta),k)\cong  X\times\conf(\R^N,k),
\]
where $N=\rank\xi+\rank\zeta$, and the projection map coincides with the projection on the first coordinate.
The continuous fiberwise map \eqref{diag : Morphism - 1} induced now by $\Phi\oplus \Id$ has form
\begin{equation}
	\label{diag : Morphism - 2}
	\xymatrix{
				X\times\conf(\R^N,k)\cong \conf_{p_{\xi\oplus\zeta}}(E(\xi\oplus\zeta),k)\ar[rr]^-{(\Phi\oplus\Id)^{\oplus k}}\ar[dr]^-{p_1} &     & E((\eta\oplus\zeta)^{\oplus k})\ar[dl]\\
                         &X&          
}
\end{equation}
where $p_1$ denotes the projection on the first factor.
Typical fiber of the bundle $\conf_{p_{\xi\oplus\zeta}}(E(\xi\oplus\zeta),k)$ is homeomorphic to the configuration space $\conf(\R^N,k)$ and is equipped with the free action of the symmetric group $\Sym_k$.
If $T=\rank\eta + \rank\zeta$, then the fiber of $(\eta\oplus\zeta)^{\oplus k}$ is a real $\Sym_k$-representation $(\R^T)^{\oplus k}$ where the action is given by permutation of factors in the $k$-fold direct sum.
The actions on the fibers induce $\Sym_k$-actions on $\conf_{p_{\xi\oplus\zeta}}(E(\xi\oplus\zeta),k)$ and $ E((\eta\oplus\zeta)^{\oplus k})$.
Again, the fiberwise map $(\Phi\oplus \Id)^{\oplus k}$, as well as its fiberwise restrictions, are $\Sym_k$-equivariant maps.

\medskip
Next, consider vector bundle monomorphism $\Delta\colon\eta\oplus\zeta\longrightarrow (\eta\oplus\zeta)^{\oplus k}$, that is a continuous fiberwise map linear on each fiber, given by the diagonal embedding.
It is an $\Sym_k$-equivariant map, and its image $\alpha:=\im\Delta$ is an $\Sym_k$-invariant vector subbundle of $(\eta\oplus\zeta)^{\oplus k}$.
Let  $\beta$ be the orthogonal complement of $\alpha$ in $(\eta\oplus\zeta)^{\oplus k}$.
Hence, $(\eta\oplus\zeta)^{\oplus k}\cong\alpha\oplus\beta$, and $\beta$ is an $\Sym_k$-invariant vector subbundle of $(\eta\oplus\zeta)^{\oplus k}$.
On the level of the fibers this decomposition becomes $(\R^T)^{\oplus k}\cong W_k^{\oplus T}\oplus\R^T$, where $\R^T$ is a trivial $\Sym_k$-representation and $W_k=\{(y_1,\ldots,y_k)\in\R^k : \sum y_i=0\}$ is a subrepresentation of $\R^k$ where the action is given by coordinate permutation.
Observe that there exists an $\Sym_k$-equivariant isomorphism $\beta\cong (\eta\oplus\zeta)\otimes_{\R}\underline{W_k}$ where, as before, $\underline{W_k}$ denotes a trivial vector bundle over $X$ with fiber $W_k$.

\medskip
Furthermore, let $\Pi\colon (\eta\oplus\zeta)^{\oplus k}\cong\alpha\oplus\beta\longrightarrow\beta$ denote the vector bundle morphism given by the projection.
The projection $\Pi$ is an $\Sym_k$-equivariant map.
Now consider the following composition $\Pi\circ (\Phi\oplus\Id)^{\oplus k}$ of continuous fiberwise maps:
\begin{equation*}
	\xymatrix{
				X\times\conf(\R^N,k)\cong \conf_{p_{\xi\oplus\zeta}}(E(\xi\oplus\zeta),k)\ar[rr]^-{(\Phi\oplus\Id)^{\oplus k}}\ar[drr]^-{p_1} &   &   E((\eta\oplus\zeta)^{\oplus k})\ar[r]^-{\Pi}\ar[d]  & E(\beta)\ar[dl]\\
				&& X. &
}
\end{equation*}
The composition of fiberwise maps $\Pi\circ (\Phi\oplus\Id)^{\oplus k}$ is an $\Sym_k$-equivariant map with respect to the already defined actions.
Moreover, it has the following important property:
{\em If there exists a point $x_0\in X$ such that the image of the fiber of the bundle $\conf_{p_{\xi\oplus\zeta}}(E(\xi\oplus\zeta),k)$ over $x_0$ along the fiberwise  map $\Pi\circ (\Phi\oplus\Id)^{\oplus k}$ contains zero of the fiber over $x_0$ of $\beta$, then the continuous fiberwise map $\Phi\oplus\Id$, and consequently $\Phi$, admits a $k$-multiplicity.}
Thus, if the image of every continuous $\Sym_k$-equivariant fiberwise map $\conf_{p_{\xi\oplus\zeta}}(E(\xi\oplus\zeta),k)\longrightarrow E(\beta)$ does not avoid the zero section of the vector bundle $\beta$, then a $k$-multiplicity of any continuous fiberwise map $\xi\longrightarrow\eta$ is guaranteed. 

\medskip
In order to translate this property in a more convenient language we consider the following pullback vector bundle :
\begin{equation}
	\label{diag : pullback-bundle}
	\xymatrix{
	E(p_1^*\beta)\ar[rr]^-{\Psi}\ar[d] & &E(\beta)\ar[d]\\
	X\times\conf(\R^N,k)\cong \conf_{p_{\xi\oplus\zeta}}(E(\xi\oplus\zeta),k)\ar[rr]^-{p_1}\ar[urr]^{\Pi\circ(\Phi\oplus\Id)^{\oplus k}} & &X,
}
\end{equation}
where
\begin{multline*}
E(p_1^*\beta)=\{ (x,(v_1,\ldots,v_k),u)\in X\times\conf(\R^N,k)\times E(\beta) :\\
	 p_1(x,(v_1,\ldots,v_k))=x=p_{\beta}(u)  \},
\end{multline*}
and the bundle morphism $\Phi$ is given by $(x,(v_1,\ldots,v_k),u)\longmapsto u$.
The action of the symmetric group $\Sym_k$ on $E(p_1^*\beta)$ is given by the diagonal action on the product $X\times\conf(\R^N,k)\times E(\beta)$, having in mind that the actions on $X$ and $E(\beta)$ are trivial. 

\medskip
The bundle morphism $\Pi\circ(\Phi\oplus\Id)^{\oplus k}$ induces the $\Sym_k$-equivariant section of the pullback bundle $s\colon \conf_{p_{\xi\oplus\zeta}}(E(\xi\oplus\zeta),k)\longrightarrow E(p_1^*\beta)$ defined by
\begin{equation}
	\label{eq : s}
	(x,(v_1,\ldots,v_k)) \longmapsto \big(x,(v_1,\ldots,v_k),(\Pi\circ(\Phi\oplus\Id)^{\oplus k})(x,(v_1,\ldots,v_k))\big).
\end{equation} 
Now, {\em the continuous $\Sym_k$-equivariant fiberwise map $\Pi\circ(\Phi\oplus\Id)^{\oplus k}$ hits the zero section of the vector bundle $\beta$ if and only if the $\Sym_k$-equivariant section $s$ defined in \eqref{eq : s} hits the corresponding zero section}.
Thus, the aim is to prove that every $\Sym_k$-equivariant section of the pullback bundle hits the zero section.
Indeed, this would in particular imply that the section $s$ hits the zero section, consequently $\Pi\circ(\Phi\oplus\Id)^{\oplus k}$ hits the zero section and so the continuous fiberwise map $\Phi$ admits a $k$-multiplicity.

\medskip
The vector bundle $\beta$ is isomorphic to the tensor product $(\eta\oplus\zeta)\otimes_{\R}\underline{W_k}$.
Since the tensor product and the Whitney sum are compatible with the pullback we have
\begin{equation}
	\label{eq : tensor decomposition - 1}
	E(p_1^*\beta)\cong E(p_1^*\eta\oplus p_1^*\zeta)\otimes_{\R} E(p_1^*\underline{W_k}).
\end{equation}

\medskip
The action of $\Sym_k$ on the total and on the base space of the pullback vector bundle is free.
Moreover, the projection map of $p_1^*\beta$ is an $\Sym_k$-equivariant map.
Therefore, when dividing out the $\Sym_k$ action we obtain the vector bundle
\begin{equation}
	\label{eq : tensor decomposition - 1.5}
	E(p_1^*\beta)/\Sym_k\longrightarrow \conf_{p_{\xi\oplus\zeta}}(E(\xi\oplus\zeta),k)/\Sym_k.
\end{equation}
Every $\Sym_k$-equivariant section of the pullback bundle 
\[
E(p_1^*\beta)\longrightarrow \conf_{p_{\xi\oplus\zeta}}(E(\xi\oplus\zeta),k)
\]
induces a section of the quotient bundle 
\[
E(p_1^*\beta)/\Sym_k\longrightarrow \conf_{p_{\xi\oplus\zeta}}(E(\xi\oplus\zeta),k)/\Sym_k.
\]
Furthermore, presentation \eqref{eq : tensor decomposition - 1} implies the following presentation of the vector bundle 
\eqref{eq : tensor decomposition - 1.5} as a tensor product
\begin{equation}
	\label{eq : tensor decomposition - 2}
	E(p_1^*\beta)/\Sym_k\cong E(p_1^*\eta\oplus p_1^*\zeta)/\Sym_k\otimes_{\R} E(p_1^*\underline{W_k})/\Sym_k.
\end{equation}
Note that the action of the symmetric group $\Sym_k$ on $E(p_1^*\eta\oplus p_1^*\zeta)\cong E(p_1^*(\eta\oplus\zeta))$ is induced by the trivial action on $E(\eta\oplus\zeta)$ and by the diagonal action on the base space $X\times\conf(\R^N,k)$, assuming trivial action on $X$.

\medskip
We have derived a criterion for the existence of a $k$-multiplicity for a fiberwise map between two vector bundles over the same base space.
Assuming already introduced notation we formulate the following criterion.

\begin{theorem}
\label{th : criterion}
Let $X$ be a topological space with the homotopy type of a finite CW-complex, and let $k\geq 2$ be an integer.
Let  $\xi$ and $\eta$ be vector bundles over $X$.
If the vector bundle 
\begin{multline*}
E(p_1^*((\eta\oplus(-\xi))\otimes_{\R}\underline{W_k}))/\Sym_k\cong E(p_1^*\eta\oplus p_1^*(-\xi))/\Sym_k\otimes_{\R} E(p_1^*\underline{W_k})/\Sym_k \longrightarrow  \\ 
\conf_{p_{\xi\oplus\zeta}}(E(\xi\oplus(-\xi)),k)/\Sym_k \cong X\times\conf(\R^N,k)/\Sym_k	
\end{multline*}
does not admit a nowhere zero section, then any continuous fiberwise map $\xi\longrightarrow \eta$ admits a $k$-multiplicity. 
\end{theorem}

In the case of a fiberwise map between complex vector bundles the following analogous criterion for the existence of a $k$-multiplicity can be derived in the footsteps of the construction presented in this section. 
The only difference occurs in the place of the vector bundle $\underline{W_k}$, which will be substituted by its complexification $\underline{W_k}\otimes_{\R}\C$, which as a real vector bundle is isomorphic to $\underline{W_k}^{\oplus 2}$.

\begin{theorem}
\label{th : criterion-C}
Let $X$ be a topological space with the homotopy type of a finite CW-complex, and let $k\geq 2$ be an integer.
Let  $\xi$ and $\eta$ be complex vector bundles over $X$.
If the complex vector bundle 
\begin{multline*}
E(p_1^*((\eta\oplus(-\xi))\otimes_{\C}(\underline{W_k}\otimes_{\R}\C)))/\Sym_k\cong \\
E(p_1^*\eta\oplus p_1^*(-\xi))/\Sym_k\otimes_{\C} E(p_1^*(\underline{W_k}\otimes_{\R}\C))/\Sym_k \longrightarrow  \\ 
\conf_{p_{\xi\oplus\zeta}}(E(\xi\oplus(-\xi)),k)/\Sym_k \cong X\times\conf(\C^N,k)/\Sym_k	
\end{multline*}
does not admit a nowhere zero section, then any continuous fiberwise map $\xi\longrightarrow \eta$ admits a $k$-multiplicity. 
\end{theorem}

\section{From a continuous map to a multiplicity of fiberwise map}

In this section we relate the existence of a local $k$-multiplicity for continuous maps between two Riemannian manifolds with the existence of a $k$-multiplicity for a continuous fiberwise map between appropriately constructed vector bundles. 

\medskip
Let $M$ be an $m$-dimensional smooth closed manifold, let $N$ be an $n$-dimensional smooth manifold, and let $k\geq 2$ be an integer. 
A continuous map $f\colon M\longrightarrow N$ {\em admits a $k$-multiplicity} if there exist $k$ pairwise distinct points $x_1,\ldots,x_k$ on $M$ such that $f(x_1)=\cdots=f(x_k)$.
For example, the existence of a $2$-multiplicity for any continuous map $M\longrightarrow N$ between manifolds $M$ and $N$ means that the manifold $M$ cannot be embedded into the manifold $N$.

\medskip
Let us assume that in addition manifold $M$ is equipped with some Riemannian metric.
A continuous map $f\colon M\longrightarrow N$ \emph{admits a local $k$-multiplicity} if, for every real number $\omega >0$, there exist $k$ pairwise distinct points $x_1,\ldots,x_k$ in $M$ such that 
\[
f(x_1)=\cdots=f(x_k)
\qquad\text{and}\qquad
\diam\{x_1,\ldots,x_k\}<\omega.
\]
Note that in the case when $M$ is compact this definition does not depend on the choice of a Riemannian metric.

\medskip
For our methods to work smoothly we need to make further assumptions on manifolds $M$ and $N$.
Let us assume that $M$ is a compact Riemannian manifold and that $N$ is also a Riemannian manifold with positive injectivity radius, denoted by $\rho(N)>0$.
The compactness of $M$ implies that injectivity radius $\rho(M)$ of $M$ is also positive.
Furthermore, let us fix a real number $\omega >0$.
The method we present can be in principle also applied when manifold $M$ is open, but additional assumptions need to be imposed.

Now, let $D_{\delta}(\tau M)$ and $D_{\varepsilon}(\tau N)$ denote open disc subbundles (of disc radius $\delta$ and $\varepsilon$, respectively) of the tangent bundles $\tau M$ and $\tau N$.  
The exponential maps associated to $M$ and $N$ are denoted by 
\[
\exp_M\colon E(\tau M)\longrightarrow M\times M
\qquad\text{and}\qquad
\exp_N\colon E(\tau N)\longrightarrow N\times N.
\]
Then, for a continuous map $f\colon M\longrightarrow N$ and for every $0<\varepsilon<\rho(N)$, there exists $0<\delta<\min \{\rho(M),\tfrac{\omega}2\}$ such that
\[
(f\times f)\circ \exp_M(E(D_{\delta}(\tau M)))\subseteq \exp_N(E(D_{\varepsilon}(\tau N))).
\]
Since the exponential map is injective inside the injectivity radius the following composition is well defined
\begin{equation}
	\label{eq : **}
	\Phi':=\exp_N^{-1}\circ\, (f\times f) \circ \exp_M\colon E(D_{\delta}(\tau M)) \longrightarrow E(D_{\varepsilon}(\tau N)),
\end{equation}
and illustrated in the diagram below:
\begin{equation*}
	\xymatrix{
				E(D_{\delta}(\tau M))\ar[r]^-{\exp_M}\ar[drr]_-{\Phi'} &  M\times M\ar[r]^{f\times f}  & N\times N\\
				&  & E(D_{\varepsilon}(\tau N))\ar[u]_{\exp_N}.
}
\end{equation*}
The map $\Phi'$ is a continuous fiberwise map covering the continuous map $f\colon M\longrightarrow N$. 
It plays a role of a differential map $df$ that cannot be considered in this case since $f$ is not assumed to be smooth.
Furthermore, the continuous fiberwise map $\Phi'$ induces a continuous fiberwise map $\Phi\colon E(D_{\delta}(\tau M))\longrightarrow f^* E(D_{\varepsilon}(\tau N))$ between the bundle $E(D_{\delta}(\tau M))$ and the pullback bundle $f^* E(D_{\varepsilon}(\tau N))$ by
\begin{equation}
	\label{eq : *}
	\Phi (x,v) = (x,\Phi'(x,v)),
\end{equation}
and the following diagram commutes
\begin{equation*}
	\xymatrix{
				E(D_{\delta}(\tau M))\ar[r]^-{\Phi}\ar[dr] \ar@/^1.5pc/[rr]^{\Phi'} & f^* E(D_{\varepsilon}(\tau N))\ar[r]\ar[d] & E(D_{\varepsilon}(\tau N))\ar[d]\\
				                        &  M\ar[r]^-{f}                          & N .             
}
\end{equation*}

\medskip
Now we prove a theorem that relates the existence of a $k$-multiplicity of an appropriately defined continuous fiberwise map with the existence of $k$-multiplicity for a continuous map $f\colon M\longrightarrow N$.
In the following we use already introduced notations.

\begin{theorem}
	\label{th : from map to fiberwise map depending on f}
	Let $k\geq 2$ be an integer, let $M$ be a compact Riemannian manifold, and let $N$ be a Riemannian manifolds with positive injectivity radius.
	If, for a continuous map $f\colon M\longrightarrow N$, every continuous fiberwise map $E(\tau M)\longrightarrow f^* E(\tau N)$ admits a $k$-multiplicity, then the map $f$ admits a local $k$-multiplicity.
\end{theorem}

\begin{proof}
	Let us fix $\omega>0$ and $0<\varepsilon<\rho(N)$.
	There exists $0<\delta<\min \{\rho(M),\tfrac{\omega}2\}$ such that
	\[
		(f\times f)\circ \exp_M(E(D_{\delta}(\tau M)))\subseteq \exp_N(E(D_{\varepsilon}(\tau N))).
	\]
	Now we can construct the fiberwise map $\Phi\colon E(D_{\delta}(\tau M))\longrightarrow f^* E(D_{\varepsilon}(\tau N))$, as defined in \eqref{eq : *}.

	There are fiberwise homeomorphisms 
	\[
	E(\tau M)\cong E(D_{\delta}(\tau M))\quad\text{and}\quad E(\tau N)\cong E(D_{\varepsilon}(\tau N).
	\] 
	Thus the assumption that every fiberwise map $E(\tau M)\longrightarrow  E(f^*\tau N)$ admits a $k$-multiplicity implies that  every continuous fiberwise map 
	\[
	D_{\delta}(\tau M)\longrightarrow  E(f^*D_{\varepsilon}(\tau N)),
	\]
	also admits a $k$-multiplicity, for appropriate choice of $\delta$ and $\varepsilon$.
	Consequently, the same is true for the map $\Phi$.
	Hence, there exists a point $x\in M$ and $k$ pairwise distinct vectors $v_1,\ldots,v_k\in T_xM$ of the norm less than $\delta$, such that
	\[
	\Phi(x,v_1)=\cdots=\Phi(x,v_k).
	\]
	By the definition $\Phi (x,v) = (x,\Phi'(x,v))$ and consequently
	\[
	\Phi'(x,v_1)=\cdots=\Phi'(x,v_k).
	\]
	Furthermore, in \eqref{eq : **}, we have defined that $\Phi'=\exp_N^{-1}\circ\, (f\times f) \circ \exp_M$ and so $\exp_N\circ\,\Phi'= (f\times f) \circ \exp_M$ implying
	\[
	(f\times f) \circ \exp_M(x,v_1)=\cdots=(f\times f) \circ \exp_M(x,v_k).
	\]
	Let $y_i$ denote the point on the geodesic, that starts at $x$ in direction $v_i$, on length $\norm{v_i}$ from $x$, $i\in\{1,\ldots,k\}$.
	Then the previous equalities imply that
	\[
	(f(x),f(y_1))=\cdots=(f(x),f(y_k)) \quad\Longrightarrow\quad f(y_1)=\cdots=f(y_k).
	\]
	Since $v_1,\ldots,v_k$ are pairwise distinct with the norm less then injectivity radius of $M$ we have that $y_1,\ldots,y_k$ are also pairwise distinct.
	Moreover, since $\delta<\tfrac{\omega}2$ we have that $\diam\{y_1,\ldots,y_k\}<\omega$.
	Therefore, $f$ admits a local $k$-multilicity.
\end{proof}

Combining Theorem~\ref{th : from map to fiberwise map depending on f} we just proved with Theorem~\ref{th : criterion} we get the following criterion for the existence of local $k$-multiplicity of the given continuous map $f\colon M\longrightarrow N$.

\begin{theorem}
	\label{th : criterion for continuous maps}
	Let $k\geq 2$ be an integer, let $M$ be a compact Riemannian manifold, and let $N$ be a Riemannian manifold with positive injectivity radius.
	If, for a continuous map $f\colon M\longrightarrow N$, the vector bundle 
	\begin{multline*}
E\big(p_1^*((f^*\tau N\oplus(-\tau M))\otimes_{\R}\underline{W_k})\big)/\Sym_k\cong \\
E\big(p_1^*(f^*\tau N\oplus(-\tau M))\big)/\Sym_k\otimes_{\R} E(p_1^*\underline{W_k})/\Sym_k \longrightarrow  \\ 
M\times\conf(\R^{\dim\tau M+\dim(-\tau M)},k)/\Sym_k	
\end{multline*}
does not admit a nowhere zero section, then the map $f$ admits a local $k$-multiplicity. 
\end{theorem}

Again, like in the case of Theorem \ref{th : criterion-C}, we can derive a criterion for the existence of local $k$-multiplicity for a given continuous map $f\colon M\longrightarrow N$ between, now, complex manifolds $M$ and $N$. 
As already explained, the only difference is in the place of the vector bundle $\underline{W_k}$, which is substituted by its complexification $\underline{W_k}\otimes_{\R}\C$.

\begin{theorem}
	\label{th : criterion for continuous maps-C}
	Let $k\geq 2$ be an integer, let $M$ be a compact complex Riemannian manifold, and let $N$ be a complex Riemannian manifold with positive injectivity radius.
	If, for a continuous map $f\colon M\longrightarrow N$, the vector bundle 
\begin{multline*}
E\big(p_1^*((f^*\tau N\oplus(-\tau M))\otimes_{\C}(\underline{W_k}\otimes_{\R} \C)\big)/\Sym_k\cong \\
E\big(p_1^*(f^*\tau N\oplus(-\tau M))\big)/\Sym_k\otimes_{\C} E(p_1^*(\underline{W_k}\otimes_{\R} \C))/\Sym_k \longrightarrow  \\ 
M\times\conf(\C^{\dim\tau M+\dim(-\tau M)},k)/\Sym_k	
\end{multline*}
does not admit a nowhere zero section, then the map $f$ admits a local $k$-multiplicity. 
\end{theorem}

\section{Proofs of the main results}

\subsection{Proof of Theorem \ref{th : main 0}}
Let us fix a continuous map $f\colon M\longrightarrow N$ and assume that:
\begin{compactitem}[\ \ $\circ$]
	\item  $k\geq 2$ is a power of two,
	\item  $M$ is a compact $m$-dimensional smooth manifold,
	\item  $N$ is an $n$-dimensional smooth manifold,
	\item  $r:=2m-1-s\leq 2m-1$, and 
	\item $u_r(f^*\tau N\oplus(-\tau M))=\det (w_{m+n-r-i+j}(f^*\tau N\oplus(-\tau M)))_{1\leq i,j\leq k-1}\neq 0$.
\end{compactitem}
Since $M$ and $N$ are smooth manifolds according to the work of Green \cite{Green1978} they can be equipped with a Riemannian metric in such a way that corresponding injectivity radii are positive.

Now, according to Theorem \ref{th : criterion for continuous maps} a continuous map $f$ admits a local $k$-multiplicity if the vector bundle
\begin{multline}
\label{eq : bundle-01}
E\big(p_1^*(f^*\tau N\oplus(-\tau M))\big)/\Sym_k\otimes_{\R} E(p_1^*\underline{W_k})/\Sym_k  \longrightarrow \\
M\times\conf(\R^{\dim TM+\dim(-TM)},k)/\Sym_k	
\end{multline}
does not admit a nowhere zero section.
Thus, it suffices to prove that the Euler class, or the top Stiefel--Whitney class of the vector bundle \eqref{eq : bundle-01} does not vanish.

The Whitney embedding theorem applied on the Riemannian manifold $M$ implies that $M$ can be embedded into $\R^{2m}$.
Consequently, we can assume that the inverse bundle $-\tau M$ is the normal bundle of the existing embedding of $M$ into $\R^{2m}$, and therefore an $m$-dimensional vector bundle.
Thus, the bundle \eqref{eq : bundle-01} becomes
\begin{multline}
\label{eq : bundle-02}
E\big(p_1^*(f^*\tau N\oplus(-\tau M))\big)/\Sym_k\otimes_{\R} E(p_1^*\underline{W_k})/\Sym_k  \longrightarrow \\
M\times\conf(\R^{2m},k)/\Sym_k .
\end{multline}
In order to prove the theorem {\em we will show that the mod $2$ Euler class, which is the $((m+n)(k-1))$-st Stiefel--Whitney class of the vector bundle \eqref{eq : bundle-02} does not vanish}.

To simplify notation let us denote by $\xi$ the vector bundle $E(p_1^*\underline{W_k})/\Sym_k$ and by $\eta$ the bundle $E\big(p_1^*(f^*\tau N\oplus(-\tau M))\big)/\Sym_k$.
With the notation just introduced we will compute the $((m+n)(k-1))$-st Stiefel--Whitney class
\[
w:=w_{(m+n)(k-1)}(\eta\otimes_{\R}\xi)
\]
of the vector bundle $\eta\otimes_{\R}\xi$.
The cohomology class $w$ lives in the following cohomology group that decomposes into the direct sum by the K\"unneth formula \cite[Thm.\,VI.1.6]{Bredon1993}:
\begin{multline*}
H^{(m+n)(k-1)}(M\times \conf(\R^{2m},k)/\Sym_k;\F_2)\cong \\
\bigoplus_{i=0}^{(m+n)(k-1)}H^{i}(M;\F_2)\otimes_{\F_2} H^{(m+n)(k-1)-i}( \conf(\R^{2m},k)/\Sym_k;\F_2).	
\end{multline*}

\noindent
Now we prove that the Stiefel--Whitney class $w$ does not vanish is several steps.

\subsubsection{~}
We first analyze the vector bundle $\xi=E(p_1^*\underline{W_k})/\Sym_k$.
Consider the vector bundle
\[
 \zeta \colon  W_k   \longrightarrow   \conf(\R^{2m},k)\times_{\Sym_k}W_k   \longrightarrow   \conf(\R^{2m},k)/\Sym_k,
\]
and the projection map 
\[
p_2 \colon M\times  \conf(\R^{2m},k)/\Sym_k \longrightarrow  \conf(\R^{2m},k)/\Sym_k.
\]
Then there is an isomorphism of vector bundles $\xi\cong p_2^{*}\zeta$.
Consequently, by the naturality property of the Stiefle--Whitney classes \cite[Ax.\,2, p.\,37]{Milnor1974} and the fact that $p_2$ is a projection onto the second factor we have that
\begin{equation}
	\label{formula-2}
	w_{i}(\xi) = p_2^*(w_{i} (\zeta))=1\otimes_{\F_2} w_{i} (\zeta)
	 \in H^{0}(M;\F_2)\otimes_{\F_2} H^{i}( \conf(\R^{2m},k)/\Sym_k;\F_2),
\end{equation}
for any integer $i\geq 0$.
In particular, we have that $w_{k-1}(\xi)= 1\otimes_{\F_2} w_{k-1} (\zeta)$.
According to \cite[Lem.\,8.14]{Blagojevic2105}, and as in \cite[eq.\,(2), p.\,7]{Blagojevic2016}, the following equivalence holds
\begin{equation}
	\label{formula-4}
	w_{k-1}(\xi)^{j}\neq 0 \quad\qquad\text{if and only if}\quad\qquad 0\leq j\leq 2m-1. 
\end{equation}
More information about characteristic classes of the vector bundle $\xi$ is given in the following lemma, \cite[Cor.\,2.16]{Blagojevic2016}.
\begin{lemma}
	\label{lem : comp}
	Consider a matrix $[j_{r,s}]_{1\leq r\leq t,1\leq s\leq k-1}$ of non-negative integers with pairwise distinct rows.
	Assume that, for some $0\leq j\leq 2m-1$ and each $1\leq r\leq t$,
	\[
 		\sum_{s=1}^{k-1}s\cdot j_{r,s} = (k-1)j.
	\]
	Then, for some $\lambda_1,\ldots,\lambda_t\in\F_2$,
	\[
 	\sum_{r=1}^{t}\lambda_r\cdot w_1(\xi)^{j_{r,1}}\cdots w_{k-2}(\xi)^{j_{r,k-2}}w_{k-1}(\xi)^{j_{r,k-1}}=w_{k-1}(\xi)^{j}
	\]
	if and only if there exists a (unique) $r_0 \in \{1,2, \ldots , t\}$ such that
	\begin{compactitem}
 		\item $\lambda_r=0$ if and only if $r\neq r_0$, and
 		\item $j_{r_0,1}=\cdots=j_{r_0,k-2}=0,\,j_{r_0,k-1}=j$.
	\end{compactitem}
\end{lemma}

\subsubsection{~}
Second, consider the vector bundle $\eta=E\big(p_1^*(f^*\tau N\oplus(-\tau M))\big)/\Sym_k$.
The projection map $p_1\colon M\times \conf(\R^{2m},k)\longrightarrow M$ factors as follows
\[
M\times \conf(\R^{2m},k)\overset{\pi_2}{\longrightarrow} M\times \conf(\R^{2m},k)/\Sym_k\overset{\pi_1}{\longrightarrow} M.
\]
Thus, $\eta\cong\pi_1^*(f^*\tau N\oplus(-\tau M)))$.
Again, the naturality property for the Stiefel--Whitney classes implies that
\begin{equation}
	\label{formula-5}
	w_{i}(\eta)=\pi_1^*(w_i(f^*\tau N\oplus(-\tau M)))= w_{i}(f^*\tau N\oplus(-\tau M))\otimes_{\F_2} 1,
\end{equation}
for every integer $i\geq 0$.
According to one of theorem's assumptions
\begin{equation}
	\label{formula-6}
u_r(f^*\tau N\oplus(-\tau M))=\det (w_{m+n-r-i+j}(f^*\tau N\oplus(-\tau M)))_{1\leq i,j\leq k-1}\neq 0.
\end{equation}

\subsubsection{~}
Now we compute $w=w_{(m+n)(k-1)}(\eta\otimes_{\R}\xi)$.
For this we use the following known formula \cite[Thm.\,1]{Thomas1959}, \cite[Pr.\,7-C]{Milnor1974} for the total Stiefel--Whitney class  of the tensor product of vector bundles
\begin{equation}
	\label{formula-7}
	w(\eta\otimes_{\R}\xi)=P(w_1(\eta),\ldots,w_{m+n}(\eta),w_1(\xi),\ldots,w_{k-1}(\xi)).
\end{equation}
Here $P$ denotes the polynomial 
\[
P(\sigma_1,\ldots,\sigma_{m+n},\sigma'_1,\ldots,\sigma'_{k-1})=\prod_{i=1}^{m+n} \prod_{j=1}^{k-1}(1+a_i+b_j)
\]
that belongs to the ring of symmetric polynomials
\[
\F_2[\sigma_1,\ldots,\sigma_{m+n},\sigma'_1,\ldots,\sigma'_{k-1}]=\F_2[a_1,\ldots,a_{m+n},b_1,\ldots,b_{k-1}]^{\Sym_{m+n}\times \Sym_{k-1}}.
\]
Here $\sigma_1,\ldots,\sigma_{m+n}$ stand for the elementary symmetric polynomials in variables $a_1,\ldots,a_{m+n}$, and $\sigma'_1,\ldots,\sigma'_{k-1}$ are the elementary symmetric polynomials in variables $b_1,\ldots,b_{k-1}$.

In order to compute the class $w$ we will identify the $((m+n)(k-1))$-homogenous part of the polynomial $P$ expressed in terms of elementary symmetric polynomials $\sigma_1,\ldots,\sigma_{m+n}$ and $\sigma'_1,\ldots,\sigma'_{k-1}$ that correspond to the Stiefel--Whithey classes of $\eta$ and $\xi$.

\subsubsection{~}
\label{subsec : polynomial}
The $((m+n)(k-1))$-homogenous part $P_{(m+n)(k-1)}$ of the polynomial $P$, that computes the Stiefel--Whitney class $w_{(m+n)(k-1)}(\eta\otimes_{\R}\xi)$, can be expressed in the following form
\begin{equation}
\label{formula-7.1}
P_{(m+n)(k-1)} = \prod_{i=1}^{m+n} \prod_{j=1}^{k-1}(a_i+b_j).	
\end{equation}
According to a dual version of the Cauchy identity \cite[Eq.\,(0.11')]{Macdonald1992} 
\begin{equation}
\label{formula-8}
P_{(m+n)(k-1)} = \sum_{\lambda}s_{\lambda}(a) s_{\widehat{\lambda}'}(b),	
\end{equation}
where
\begin{compactitem}[\ \ $\circ$]
\item $\lambda=(\lambda_1,\ldots,\lambda_{m+n})$ is a partition, that means $\lambda_1\geq\cdots\geq\lambda_{m+n}\geq 0$,
\item $\lambda_1\leq k-1$,
\item $\widehat{\lambda}=(k-1-\lambda_{m+n},k-1-\lambda_{m+n-1},\ldots,k-1-\lambda_1)$,
\item $\widehat{\lambda}'$ is the conjugate partition of $\widehat{\lambda}$, and
\item $s_{\lambda}(a)=s_{\lambda}(a_1,\ldots,a_{m+n})$ is the Schur function associated to the partition $\lambda$.
\end{compactitem}
The Schur function $s_{\lambda}(a)$ is a symmetric polynomial in $a_1,\ldots,a_{m+n}$ and is defined by:
\[
s_{\lambda}(a_1,\ldots,a_{m+n})=\frac{\det \big(a_i^{\lambda_j+m+n-j}\big)_{1\leq i,j\leq m+n}}{\det \big(a_i^{m+n-j}\big)_{1\leq i,j\leq m+n}}.
\]
On the other hand, the N\"agelsback--Kosta formula \cite[Eq.\,(0.3)]{Macdonald1992} gives a presentation of the Schur function $s_{\lambda}(a)$ in terms of elementary symmetric functions $\sigma_1,\ldots,\sigma_{m+n}$ as follows:
\begin{eqnarray}
\label{formula_NK}
s_{\lambda}(a_1,\ldots,a_{m+n}) &=& \det\big( \sigma_{\lambda_i'-i+j} \big)_{1\leq i,j\leq t}\nonumber	\\
&=& \left|{\small
\begin{array} {lllll}
	\sigma_{\lambda_1'}  & \sigma_{\lambda_1'+1} & \sigma_{\lambda_1'+2} & \cdots & \sigma_{\lambda_1'+t-1}\\
	\sigma_{\lambda_2'-1}  & \sigma_{\lambda_2'} & \sigma_{\lambda_2'+1} & \cdots & \sigma_{\lambda_2'+t-2}\\
	  \cdots &  \cdots & \cdots  & \cdots &  \\
	 \sigma_{\lambda_t'-t+1}  & \sigma_{\lambda_t'-t+2} & \sigma_{\lambda_t'-t+3} & \cdots & \sigma_{\lambda_t'} 
\end{array}}
\right|
\end{eqnarray}
where $t$ is the length of the conjugate partition $\lambda'$.
Here we assume that $\sigma_0=1$, and $\sigma_i=0$ for $i<0$ or $i>m+n$.

\subsubsection{~}
In the next step, having in mind Lemma~\ref{lem : comp}, we want to identify all the Schur functions $s_{\widehat{\lambda}'}(b)$ in the formula \eqref{formula-8} that have a power of the elementary symmetric polynomial $\sigma_{k-1}'$ in their presentation.
Recall that $\sigma_{k-1}'$ corresponds to the Stiefel--Whitney class $w_{k-1}(\xi)$.

From the N\"agelsback--Kosta formula \eqref{formula_NK} the Schur function $s_{\widehat{\lambda}'}(b)$ has a power of $(\sigma_{k-1}')^t$ in its presentation if and only if
\begin{eqnarray*}
\widehat{\lambda}=(\underbrace{k-1,\ldots,k-1}_{t}) &	\Longleftrightarrow &  \lambda=(\underbrace{k-1,\ldots,k-1}_{m+n-t}) \\
&	\Longleftrightarrow & \  \lambda'=(\underbrace{m+n-t,\ldots,m+n-t}_{k-1}).
\end{eqnarray*}
In this case $s_{\widehat{\lambda}'}(b)= (\sigma_{k-1}')^t$, and 
\[
s_{\lambda}(a_1,\ldots,a_{m+n})=\hspace{-2pt} \left|{\small
\begin{array} {lllll}
	\sigma_{m+n-t}  & \sigma_{m+n-t+1} & \sigma_{m+n-t+2} & \cdots & \sigma_{m+n-t+k-2}\\
	\sigma_{m+n-t-1}  & \sigma_{m+n-t} & \sigma_{m+n-t+1} & \cdots & \sigma_{m+n-t+k-3}\\
	  \cdots &  \cdots & \cdots  & \cdots &  \\
	 \sigma_{m+n-t-k+2}  & \sigma_{m+n-t-k+3} & \sigma_{m+n-t-k+4} & \cdots &\sigma_{m+n-t} 
\end{array}}
\right|.
\]
Thus, we have the following presentation:
\begin{multline}
\label{formula-8.2}
P_{(m+n)(k-1)}
= \prod_{i=1}^{m+n} \prod_{j=1}^{k-1}(a_i+b_j)	= \sum_{\lambda}s_{\lambda}(a) s_{\widehat{\lambda}'}(b)\\
= \sum_{t=0}^{m+n} \left|{\small
\begin{array} {llll}
	\sigma_{m+n-t}  & \sigma_{m+n-t+1} &  \cdots & \sigma_{m+n-t+k-2}\\
	\sigma_{m+n-t-1}  & \sigma_{m+n-t} &  \cdots & \sigma_{m+n-t+k-3}\\
	  \cdots &  \cdots  & \cdots &  \\
	 \sigma_{m+n-t-k+2}  & \sigma_{m+n-t-k+3}  & \cdots &\sigma_{m+n-t} 
\end{array}}
\right|(\sigma'_{k-1})^t + \sum_{j\in J}\alpha_j\beta_j,	
\end{multline}
where $\alpha_j\in \F_2[\sigma_1,\ldots,\sigma_m]$ and $\beta_j\in \F_2[\sigma'_1,\ldots,\sigma'_{k-1}]$ are monomials, and {\bf no} $\beta_j$ is a power of $\sigma'_{k-1}$.
Now combining \eqref{formula-7} and \eqref{formula-8.2} we get that
\begin{multline}
\label{formula-9}
	w_{(m+n)(k-1)}(\eta\otimes_{\R}\xi) = \\
	\sum_{t=0}^{m+n}  
	\left|{\small 
\begin{array} {llll}
	w_{m+n-t}(\eta)  & w_{m+n-t+1}(\eta) &  \cdots & w_{m+n-t+k-2}(\eta)\\
	w_{m+n-t-1}(\eta)  & w_{m+n-t}(\eta) &  \cdots & w_{m+n-t+k-3}(\eta)\\
	  \cdots &  \cdots  & \cdots &  \\
	w_{m+n-t-k+2}(\eta)  & w_{m+n-t-k+3}(\eta)  & \cdots & w_{m+n-t}(\eta) 
\end{array}}
\right|
 w_{k-1}(\xi)^t +\\
  \sum_{j\in J}\alpha_j \beta_j,
\end{multline}
where $\alpha_j\in \F_2[w_1(\eta),\ldots,w_m(\eta)]$ and $\beta_j\in \F_2[w_1(\xi),\ldots,w_{k-1}(\xi)]$ are non-constant monomials, and {\bf no} $\beta_j$ is a power of $w_{k-1}(\xi)$.

Let us introduce notation
\[
u_t(\alpha):=
\left|{\small 
\begin{array} {llll}
	w_{m+n-t}(\alpha)  & w_{m+n-t+1}(\alpha) &  \cdots & w_{m+n-t+k-2}(\alpha)\\
	w_{m+n-t-1}(\alpha)  & w_{m+n-t}(\alpha) &  \cdots & w_{m+n-t+k-3}(\alpha)\\
	  \cdots &  \cdots  & \cdots &  \\
	w_{m+n-t-k+2}(\alpha)  & w_{m+n-t-k+3}(\alpha)  & \cdots & w_{m+n-t}(\alpha) 
\end{array}}
\right|
,
\]
where $\alpha$ is an $(m+n)$-dimensional vector bundle.
It is important to keep in mind that $\deg (u_t)=(m+n-t)(k-1)$.
Then the formula \eqref{formula-9} becomes:
\begin{equation}
\label{formula-9-1}
	w_{(m+n)(k-1)}(\eta\otimes_{\R}\xi) = 
	\sum_{t=0}^{m+n}  
	u_t(\eta) w_{k-1}(\xi)^t +
  \sum_{j\in J}\alpha_j \beta_j,	
\end{equation}
where $\alpha_j\in \F_2[w_1(\eta),\ldots,w_m(\eta)]$ and $\beta_j\in \F_2[w_1(\xi),\ldots,w_{k-1}(\xi)]$ are non-constant monomials, and {\bf no} $\beta_j$ is a power of $w_{k-1}(\xi)$.

\subsubsection{~}
Now, having in mind \eqref{formula-2} and \eqref{formula-5} we transform expression \eqref{formula-9-1} as follows:
\begin{equation}
\label{formula-10}
	w_{(m+n)(k-1)}(\eta\otimes_{\R}\xi) = 
	\sum_{t=0}^{m+n}  
	u_t(f^*\tau N\oplus(-\tau M)) \otimes_{\F_2} w_{k-1}(\zeta)^t +
  \sum_{j\in J}\alpha_j' \otimes_{\F_2} \beta_j',	
\end{equation}
where 
\[
\alpha_j=\alpha'_j\otimes_{\F_2} 1
\qquad\text{and}\qquad
\beta_j=1\otimes_{\F_2}\beta'_j,
\]
for some $\alpha'_j\in H^{\geq 1}(M;\F_2)$, and some $\beta'_j\in H^{\geq 1}(\conf(\R^{2m},k)/\Sym_k;\F_2)$ not a power of $w_{k-1}(\zeta)$.

Next, recall that by an assumption of the theorem there exists $r\leq 2m-1$ with the property 
\[
u_r(f^*\tau N\oplus(-\tau M))=\det (w_{\dim M+\dim N -r-i+j}(f^*\tau N\oplus(-\tau M)))_{1\leq i,j\leq k-1}\neq 0.
\]
Hence, consider the projection induced by the K\"unneth formula decomposition
\begin{multline*}
\Lambda\colon H^{(m+n)(k-1)}(M\times \conf(\R^{2m},k)/\Sym_k;\F_2) \longrightarrow \\
H^{(m+n-r)(k-1)}(M;\F_2)\otimes_{\F_2} H^{r(k-1)}( \conf(\R^{2m},k)/\Sym_k;\F_2).	
\end{multline*}
Then from \eqref{formula-10} we get that
\begin{equation}
	\label{formula-12}
	\Lambda(w) =  u_r(f^*\tau N\oplus(-\tau M)) \otimes_{\F_2} w_{k-1}(\zeta)^{r} + \sum_{j\in J''}\alpha''_j\otimes_{\F_2}  \beta''_j\end{equation}
where $\deg \alpha''_j =(m+n-r)(k-1)$, $\deg \beta''_j= r(k-1)$ for $r\geq 1$, and $\beta''_j=0$ for $r=0$.
Moreover, {\bf no} $\beta''_j$ is equal to $w_{k-1}(\zeta)^{r}$.

Since $u_r(f^*\tau N\oplus(-\tau M))\neq 0$ and by \eqref{formula-4} we have $w_{k-1}(\zeta)^{r}\neq 0$ (because $r\leq 2m-1$), we have that the first summand in the formula \eqref{formula-12} does not vanish.
On the other hand by \cite[Cor.\,2.16]{Blagojevic2016} we get that first and second summand in \eqref{formula-12} do not coincide, or since we are working with coefficients in $\F_2$, they do no cancel.
Thus, $\Lambda(w)\neq 0$ and consequently the mod $2$ Euler class $w$ of the vector bundle $\eta\otimes_{\R}\xi$ does not vanish.
This concludes the proof of the theorem.

\subsection{Proof of Theorem \ref{th : main 0-C}}
Let us fix a continuous map $f\colon M\longrightarrow N$ and assume that:
\begin{compactitem}[\ \ $\circ$]
	\item  $k\geq 2$ is an odd prime,
	\item  $M$ is a compact $m$-dimensional smooth almost complex manifold,
	\item $-\tau M$ is an $m'$-dimensional  complex vector bundle,
	\item  $N$ is an $n$-dimensional smooth complex manifold,
	\item  $r\leq m+m'-1$, and 
	\item $v_r(f^*\tau N\oplus(-\tau M))=\det (c_{m'+n-r-i+j}(f^*\tau N\oplus(-\tau M)))_{1\leq i,j\leq k-1}\neq 0$.
\end{compactitem}
Here we put $r = m+m'-1-s$. Furthermore, the Chern classes we work with in this proof are considered mod $k$. Since $M$ and $N$ are smooth manifolds they can be equipped with a Riemannian metric in such a way that corresponding injectivity radii are positive, see  \cite{Green1978}.

From Theorem \ref{th : criterion for continuous maps-C} we have that the continuous map $f$ admits a local $k$-multiplicity if the complex vector bundle
\begin{multline}
\label{eq : bundle-01-C}
E\big(p_1^*((f^*\tau N\oplus(-\tau M))\otimes_{\C}(\underline{W_k}\otimes_{\R} \C)\big)/\Sym_k\cong \\
E\big(p_1^*(f^*\tau N\oplus(-\tau M))\big)/\Sym_k\otimes_{\C} E(p_1^*(\underline{W_k}\otimes_{\R} \C))/\Sym_k \longrightarrow  \\ 
M\times\conf(\C^{\dim\tau M+\dim(-\tau M)},k)/\Sym_k	
\end{multline}
does not admit a nowhere zero section.
It suffices to prove that the Euler class, or the top Chern class of the complex vector bundle \eqref{eq : bundle-01-C} does not vanish.

Since the inverse bundle $-\tau M$ is an $m'$-dimensional complex vector bundle the bundle \eqref{eq : bundle-01-C} becomes
\begin{multline}
\label{eq : bundle-02-C}
E\big(p_1^*(f^*\tau N\oplus(-\tau M))\big)/\Sym_k\otimes_{\C} E(p_1^*(\underline{W_k}\otimes_{\R} \C))/\Sym_k \longrightarrow  \\ 
M\times\conf(\C^{m+m'},k)/\Sym_k.
\end{multline}
To prove the theorem {\em we will demonstrate that the mod $k$ Euler class, which in this case coincides with the $((m'+n)(k-1))$-st Chern class of the complex vector bundle \eqref{eq : bundle-02-C} does not vanish}.

We simplify notation again by denoting:
\[
\xi=E(p_1^*(\underline{W_k}\otimes_{\R}\C))/\Sym_k
\qquad\text{and}\qquad
\eta=E\big(p_1^*(f^*\tau N\oplus(-\tau M))\big)/\Sym_k.
\]
Thus we need to compute the following the mod $k$ Chern class
\[
c:=c_{(m'+n)(k-1)}(\eta\otimes_{\C}\xi)
\]
of the complex vector bundle $\eta\otimes_{\C}\xi$.
The class $c$ belongs to the following cohomology group that decomposes into the direct sum by the K\"unneth formula \cite[Thm.\,VI.1.6]{Bredon1993}:
\begin{multline}
\label{eq : Kunneth-C}
H^{2(m'+n)(k-1)}(M\times \conf(\C^{m+m'},k)/\Sym_k;\F_k)\cong \\
\bigoplus_{i=0}^{2(m'+n)(k-1)}H^{i}(M;\F_k)\otimes_{\F_k} H^{2(m'+n)(k-1)-i}( \conf(\C^{m+m'},k)/\Sym_k;\F_k).
\end{multline}

\noindent
The computation of the Chern class $c$ will be done in steps.

\subsubsection{~}
First we analyze the complex vector bundle $\xi=E(p_1^*(\underline{W_k}\otimes_{\R}\C))/\Sym_k$.
Consider the complex vector bundle
\[
 \zeta \quad\colon\quad  W_k\otimes_{\R}\C   \longrightarrow   \conf(\C^{m+m'},k)\times_{\Sym_k}( W_k\otimes_{\R}\C)   \longrightarrow   \conf(\C^{m+m'},k)/\Sym_k,
\]
and the projection map 
\[
p_2 \colon M\times  \conf(\C^{m+m'},k)/\Sym_k \longrightarrow  \conf(\C^{m+m'},k)/\Sym_k.
\]
There is an isomorphism of complex vector bundles $\xi\cong p_2^{*}\zeta$.
The naturality property of the Chern classes \cite[Lem.\,14.2]{Milnor1974} and the fact that $p_2$ is a projection onto the second factor imply that
\begin{equation}
	\label{formula-2-C}
	c_{i}(\xi) = p_2^*(c_{i} (\zeta))=1\otimes_{\F_k} c_{i} (\zeta)
	 \in H^{0}(M;\F_k)\otimes_{\F_k} H^{i}( \conf(\C^{m+m'},k)/\Sym_k;\F_k),
\end{equation}
for any integer $i\geq 0$.

Next we recall some know fact about the cohomology of the unordered configuration space $\conf(\C^{m+m'},k)$ with $\F_k$ coefficients, consult for example \cite[Prop.\,5.1(iii) and Thm.\,5.2]{Cohen1976}

\begin{lemma}
\label{lem : cohomology with F_p coefficients}
Let $k$ be an odd prime, and let $m\ge1$ be an integer 
\begin{compactenum}[\rm (1)]
	\item 
	\label{prop:cohomology_of_S_p}
	 $H^*(\Sym_k;\F_k)\cong \Lambda[e]\otimes \F_k[x]$, where $\Lambda(\cdot)$ denotes the exterior algebra, $e$ is a class of degree $2k-3$ and $x$ is the Bockstein of $e$, and so a class of degree $2k-2$.
	 \item There is a monomorphism of algebras
		\[
			h^*\colon H^{\leq (2m+2m'-1)(k-1)}(\Sym_k;\F_k) \longrightarrow H^*(\conf(\C^{m+m'},k)/\Sym_k;\F_k),
		\]
	induced by the classifying map $h\colon \conf(\C^{m+m'},k)/\Sym_k\longrightarrow \B 	\Sym_k$.	
\end{compactenum}
\end{lemma}

Next, consider the complex vector bundle
\[
  \mu \quad\colon\quad  W_k\otimes_{\R}\C   \longrightarrow  \E \Sym_k \times_{\Sym_k}( W_k\otimes_{\R}\C)   \longrightarrow   \B \Sym_k.
\]
It is known that $c(\mu)=1+x$, for suitable choice of generator $x$ in the cohomology of the symmetric group.
That meaning $c_i(\mu)\neq 0$ if and only if $i\in \{0,k-1\}$, and moreover $c_{k-1}=x$.
Since $\zeta=h^*\mu$ using Lemma \ref{lem : cohomology with F_p coefficients} and naturality property of the Chern classes we get 
\begin{equation}
	\label{formula-3-C}
	c(\zeta)=1+h^*(x).
\end{equation}
Furthermore, by \cite[Thm.\,4.1]{Blagojevic2015-02} we have that
\begin{equation}
	\label{formula-4-C}
	c_{k-1}(\zeta)^{j}\neq 0 \quad\qquad\text{if and only if}\quad\qquad 0\leq j\leq m+m'-1. 
\end{equation}

\subsubsection{~}

Now we consider the complex vector bundle $\eta=E\big(p_1^*(f^*\tau N\oplus(-\tau M))\big)/\Sym_k$.
The projection map $p_1\colon M\times \conf(\C^{m+m'},k)\longrightarrow M$ factors as follows
\[
M\times \conf(\C^{m+m'},k)\overset{\pi_2}{\longrightarrow} M\times \conf(\C^{m+m'},k)/\Sym_k\overset{\pi_1}{\longrightarrow} M,
\]
and therefore $\eta\cong\pi_1^*(f^*\tau N\oplus(-\tau M)))$.
The naturality property for the Chern classes yields
\begin{equation}
	\label{formula-5-C}
	c_{i}(\eta)=\pi_1^*(c_i(f^*\tau N\oplus(-\tau M)))= c_{i}(f^*\tau N\oplus(-\tau M))\otimes_{\F_k} 1,
\end{equation}
for every integer $i\geq 0$.
According to one of theorem's assumptions
\begin{equation}
	\label{formula-6-C}
v_r(f^*\tau N\oplus(-\tau M))=\det (c_{m'+n-r-i+j}(f^*\tau N\oplus(-\tau M)))_{1\leq i,j\leq k-1}\neq 0.
\end{equation}

\subsubsection{~}
Next we make the first step in computation of $c=c_{(m'+n)(k-1)}(\eta\otimes_{\C}\xi)$.
For this we use the following known formula for the total Chern class  of the tensor product of vector bundles \cite[p.\,67]{Macdonald1995-Book}, \cite[Eq.\,(21.9)]{Bott1982}  we get
\begin{equation}
	\label{formula-7-C}
	c(\eta\otimes_{\R}\xi)=P(c_1(\eta),\ldots,c_{m+n}(\eta),c_1(\xi),\ldots,c_{k-1}(\xi)).
\end{equation}
Here $P$ is the same polynomial as in Section \ref{subsec : polynomial}:
\[
P(\sigma_1,\ldots,\sigma_{m'+n},\sigma'_1,\ldots,\sigma'_{k-1})=\prod_{i=1}^{m'+n} \prod_{j=1}^{k-1}(1+a_i+b_j)
\]
that belongs to the ring of symmetric polynomials
\[
\F_2[\sigma_1,\ldots,\sigma_{m'+n},\sigma'_1,\ldots,\sigma'_{k-1}]=\F_2[a_1,\ldots,a_{m'+n},b_1,\ldots,b_{k-1}]^{\Sym_{m'+n}\times \Sym_{k-1}}.
\]

In order to compute the class $c$ we need to identify the relevant homogenous part of the polynomial $P$.
The elementary symmetric polynomials $\sigma_1,\ldots,\sigma_{m'+n}$ and $\sigma'_1,\ldots,\sigma'_{k-1}$ correspond to the Chern classes of $\eta$ and $\xi$.
The $((m'+n)(k-1))$-homogenous part $P_{(m'+n)(k-1)}$ of the polynomial $P$, that computes the Chern class class $c$, can be expressed as in the proof of Theorem \ref{th : main 0} as follows
\begin{equation}
\label{formula-8-C}
P_{(m'+n)(k-1)} = \prod_{i=1}^{m'+n} \prod_{j=1}^{k-1}(a_i+b_j)=\sum_{\lambda}s_{\lambda}(a) s_{\widehat{\lambda}'}(b),	
\end{equation}
where
\begin{compactitem}[\ \ $\circ$]
\item $\lambda=(\lambda_1,\ldots,\lambda_{m'+n})$ is a partition, that means $\lambda_1\geq\cdots\geq\lambda_{m'+n}\geq 0$,
\item $\lambda_1\leq k-1$,
\item $\widehat{\lambda}=(k-1-\lambda_{m'+n},k-1-\lambda_{m'+n-1},\ldots,k-1-\lambda_1)$,
\item $\widehat{\lambda}'$ is the conjugate partition of $\widehat{\lambda}$, and
\item $s_{\lambda}(a)=s_{\lambda}(a_1,\ldots,a_{m'+n})$ is the Schur function associated to the partition $\lambda$.
\end{compactitem}
The N\"agelsback--Kosta formula \cite[Eq.\,(0.3)]{Macdonald1992} gives a presentation of the Schur function $s_{\lambda}(a)$ in terms of elementary symmetric functions $\sigma_1,\ldots,\sigma_{m'+n}$ as follows:
\begin{eqnarray}
\label{formula_NK-C}
s_{\lambda}(a_1,\ldots,a_{m'+n}) &=& \det\big( \sigma_{\lambda_i'-i+j} \big)_{1\leq i,j\leq t}\nonumber	\\
&=& \left|{\small
\begin{array} {lllll}
	\sigma_{\lambda_1'}  & \sigma_{\lambda_1'+1} & \sigma_{\lambda_1'+2} & \cdots & \sigma_{\lambda_1'+t-1}\\
	\sigma_{\lambda_2'-1}  & \sigma_{\lambda_2'} & \sigma_{\lambda_2'+1} & \cdots & \sigma_{\lambda_2'+t-2}\\
	  \cdots &  \cdots & \cdots  & \cdots &  \\
	 \sigma_{\lambda_t'-t+1}  & \sigma_{\lambda_t'-t+2} & \sigma_{\lambda_t'-t+3} & \cdots & \sigma_{\lambda_t'} 
\end{array}}
\right|
\end{eqnarray}
where $t$ is the length of the conjugate partition $\lambda'$.
Here we assume that $\sigma_0=1$, and $\sigma_i=0$ for $i<0$ or $i>m'+n$.

\subsubsection{~}
In the next step, having in \eqref{formula-2-C}, \eqref{formula-3-C} and \eqref{formula-4-C} , we are going to identify all the Schur functions $s_{\widehat{\lambda}'}(b)$ in the formula \eqref{formula-8-C} that have a power of the elementary symmetric polynomial $\sigma_{k-1}'$ in their presentation.
In this case, the symmetric polynomial $\sigma_{k-1}'$ corresponds to the only non-zero Chern class of positive degree $c_{k-1}(\xi)$. 

As we have already seen in the proof of Theorem \ref{th : main 0} according to the N\"agelsback--Kosta formula \eqref{formula_NK-C} the Schur function $s_{\widehat{\lambda}'}(b)$ has a power of $(\sigma_{k-1}')^t$ in its presentation if and only if
\begin{eqnarray*}
\widehat{\lambda}=(\underbrace{k-1,\ldots,k-1}_{t}) &	\Longleftrightarrow &  \lambda=(\underbrace{k-1,\ldots,k-1}_{m'+n-t}) \\
&	\Longleftrightarrow & \  \lambda'=(\underbrace{m'+n-t,\ldots,m'+n-t}_{k-1}).
\end{eqnarray*}
In this case $s_{\widehat{\lambda}'}(b)= (\sigma_{k-1}')^t$, and 
\[
s_{\lambda}(a_1,\ldots,a_{m'+n})=\hspace{-2pt} \left|{\small
\begin{array} {lllll}
	\sigma_{m'+n-t}  & \sigma_{m'+n-t+1}   & \cdots & \sigma_{m'+n-t+k-2}\\
	\sigma_{m'+n-t-1}  & \sigma_{m'+n-t}   & \cdots & \sigma_{m'+n-t+k-3}\\
	  \cdots &  \cdots    & \cdots & \cdots \\
	 \sigma_{m'+n-t-k+2}  & \sigma_{m'+n-t-k+3}   & \cdots &\sigma_{m'+n-t} 
\end{array}}
\right|.
\]
Since $c_i(\xi)=0$ for $i\notin \{0,k-1\}$, unlike in the proof of Theorem \ref{th : main 0}, we have that all Schur functions $s_{\lambda}(a_1,\ldots,a_{m'+n})$ vanish when 
\[
\lambda\neq (\underbrace{k-1,\ldots,k-1}_{m'+n-t})
\]
for some $t$.
Thus
\begin{multline}
\label{formula-8.2-C}
P_{(m'+n)(k-1)}
= \prod_{i=1}^{m'+n} \prod_{j=1}^{k-1}(a_i+b_j)	= \sum_{\lambda}s_{\lambda}(a) s_{\widehat{\lambda}'}(b)\\
= \sum_{t=0}^{m'+n} \left|{\small
\begin{array} {llll}
	\sigma_{m'+n-t}  & \sigma_{m'+n-t+1} &  \cdots & \sigma_{m'+n-t+k-2}\\
	\sigma_{m'+n-t-1}  & \sigma_{m'+n-t} &  \cdots & \sigma_{m'+n-t+k-3}\\
	  \cdots &  \cdots  & \cdots &  \\
	 \sigma_{m'+n-t-k+2}  & \sigma_{m'+n-t-k+3}  & \cdots &\sigma_{m'+n-t} 
\end{array}}
\right|(\sigma'_{k-1})^t.
\end{multline}

\subsubsection{~}
Finally, by collecting previous computation we have that
\begin{multline}
\label{formula-9-C}
	c=c_{(m'+n)(k-1)}(\eta\otimes_{\C}\xi) = \\
	\sum_{t=0}^{m'+n}  
	\left|{\small 
\begin{array} {llll}
	c_{m'+n-t}(\eta)  & c_{m'+n-t+1}(\eta) &  \cdots & c_{m'+n-t+k-2}(\eta)\\
	c_{m'+n-t-1}(\eta)  & c_{m'+n-t}(\eta) &  \cdots & c_{m'+n-t+k-3}(\eta)\\
	  \cdots &  \cdots  & \cdots &  \\
	c_{m'+n-t-k+2}(\eta)  & c_{m'+n-t-k+3}(\eta)  & \cdots & c_{m'+n-t}(\eta) 
\end{array}}
\right|   c_{k-1}(\xi)^t\\
\hspace{5pt}= \sum_{t=0}^{m'+n}  
	v_t(f^*\tau N\oplus(-\tau M))\otimes_{\F_k} c_{k-1}(\zeta)^t. 
\end{multline}
Observe that each summand in \eqref{formula-9-C} belongs to a different summand in the K\"unneth decomposition \eqref{eq : Kunneth-C} of the ambient group. 
More precisely
\begin{multline*}
	v_t(f^*\tau N\oplus(-\tau M))\otimes_{\F_k} c_{k-1}(\zeta)^t\\
	\in 
H^{2(m'+n-t)(k-1)}(M;\F_k)\otimes_{\F_k} H^{2t(k-1)}( \conf(\C^{m+m'},k)/\Sym_k.\F_k).
\end{multline*}
Hence, if one of the summands $v_t(f^*\tau N\oplus(-\tau M))\otimes_{\F_k} c_{k-1}(\zeta)^t$ does not vanish then the Chern class $c$ will also not vanish.
Since, by the theorem assumption, $v_r(f^*\tau N\oplus(-\tau M))\neq 0$ and $r\leq m+m'-1$, then according to \eqref{formula-4-C}
\[
v_r(f^*\tau N\oplus(-\tau M))\otimes_{\F_k} c_{k-1}(\zeta)^r\neq 0,
\]
and consequently $c\neq 0$. 
Thus, we completed to proof of the theorem.

\subsection{Proof of Theorem \ref{th : main 1}}
The proof of the theorem is obtained by applying Theorem \ref{th : main 0} to an arbitrary continuous map $f\colon M\longrightarrow N$.
We verify that all assumptions necessary for application of Theorem \ref{th : main 0} are met.
Let us denote by $m:=\dim M$, and by $n:=\dim N$.

First, we simplify the vector bundle $f^*\tau N\oplus(-\tau M)$.
The assumption that $N$ is parallelizable implies that the tangent bundle $\tau N$ is trivial, meaning $\tau N\cong \underline{\R^n}$ as a vector bundle over $N$.
Consequently, the pullback bundle $f^*\tau N$ is also a trivial vector bundle, denoted again by $\underline{\R^n}$, but now over $M$.
Thus,
\[
w(f^*\tau N\oplus(-\tau M))=w(\underline{\R^n}\oplus (-\tau M))=w(-\tau M)=\bar{w}(M),
\]
where $\bar{w}(M)$ denotes the total dual Stiefel--Whitney class of the tangent vector bundle $\tau M$.

Since $r:=m-s=\min\{\ell : \bar{w}_{m-\ell}(M)\neq 0\}$ and moreover $\bar{w}_{m-r}(M)^{k-1}\neq 0$ we have that
\begin{eqnarray*}
u_{n+r}(f^*\tau N\oplus(-\tau M)) &=& \\
&=& \hspace{-5pt}
\left|{\small
\begin{array} {llll}
	\bar{w}_{m-r}(M)  & \bar{w}_{m-r+1}(M) &  \cdots & \bar{w}_{m-r+k-2}(M)\\
	\bar{w}_{m-r-1}(M)  & \bar{w}_{m-r}(M) &  \cdots & \bar{w}_{m-r+k-3}(M)\\
	  \cdots &  \cdots  & \cdots &  \\
	\bar{w}_{m-r-k+2}(M)  & \bar{w}_{m-r-k+3}(M)  & \cdots & \bar{w}_{m-r}(M) 
\end{array}}
\right|\\
&=& \hspace{-5pt}
\left|{\small
\begin{array} {llll}
	\bar{w}_{m-r}(M)  & 0 &  \cdots & 0\\
	\bar{w}_{m-r-1}(M)  & \bar{w}_{m-r}(M) &  \cdots & 0\\
	  \cdots &  \cdots  & \cdots &  \\
	\bar{w}_{m-r-k+2}(M)  & \bar{w}_{m-r-k+3}(M)  & \cdots & \bar{w}_{m-r}(M) 
\end{array}}
\right|\\
&= &\bar{w}_{m-r}(M)^{k-1}\neq 0.
\end{eqnarray*}

Finally, assumption that $r\leq 2m-n-1$ implies that 
\[
n+r\leq n + 2m-n-1 = 2m-1.
\]
Now, Theorem \ref{th : main 0} implies that the continuous map $f\colon M\longrightarrow N$, that was chosen arbitrary, admits a $k$-multiplicity.
This concludes the proof of the theorem.

\subsection{Proof of Theorem \ref{th : main 1-C}}

In order to prove the theorem we apply Theorem \ref{th : main 0-C} to an arbitrary continuous map $f\colon M\longrightarrow N$.
We just verify that all assumptions necessary for application of Theorem \ref{th : main 0} are satisfied.
Let $m:=\dim M$, $m'=\dim(-\tau M)$, and $n:=\dim N$.

In this case the vector bundle $f^*\tau N\oplus(-\tau M)$ can be simplified.
Since $N$ is a parallelizable then the tangent bundle $\tau N$ is trivial, meaning $\tau N\cong \underline{\C^n}$ as a vector bundle over $N$.
Consequently, $f^*\tau N$ is also a trivial vector bundle, denoted also by $\underline{\C^n}$, but now as a complex vector bundle over $M$.
Thus,
\[
c(f^*\tau N\oplus(-\tau M))=c(\underline{\C^n}\oplus (-\tau M))=c(-\tau M)=\bar{c}(M),
\]
where $\bar{c}(M)$ denotes the $i$-th Chern class of the inverse of the tangent complex vector bundle $\tau M$.

Since by the theorem assumption $r:=m'-s=\min\{\ell : \bar{c}_{m'-\ell}(M)\neq 0\}$, and moreover $\bar{c}_{m'-r}(M)^{k-1}\neq 0$, we can evaluate the following class 
\begin{eqnarray*}
v_{n+r}(f^*\tau N\oplus(-\tau M)) &=& \\
&=& \hspace{-5pt}
\left|{\small
\begin{array} {llll}
	\bar{c}_{m'-r}(M)  & \bar{c}_{m'-r+1}(M) &  \cdots & \bar{c}_{m'-r+k-2}(M)\\
	\bar{c}_{m'-r-1}(M)  & \bar{c}_{m'-r}(M) &  \cdots & \bar{c}_{m'-r+k-3}(M)\\
	  \cdots &  \cdots  & \cdots &  \\
	\bar{c}_{m'-r-k+2}(M)  & \bar{c}_{m'-r-k+3}(M)  & \cdots & \bar{c}_{m'-r}(M) 
\end{array}}
\right|\\
&=& \hspace{-5pt}
\left|{\small
\begin{array} {llll}
	\bar{c}_{m'-r}(M)  & 0 &  \cdots & 0\\
	\bar{c}_{m'-r-1}(M)  & \bar{c}_{m'-r}(M) &  \cdots & 0\\
	  \cdots &  \cdots  & \cdots &  \\
	\bar{c}_{m'-r-k+2}(M)  & \bar{c}_{m'-r-k+3}(M)  & \cdots & \bar{c}_{m'-r}(M) 
\end{array}}
\right|\\
&= &\bar{c}_{m'-r}(M)^{k-1}\neq 0.
\end{eqnarray*}

In the last step we have that the assumption $r\leq m+m'-n-1$ implies 
\[
n+r\leq n + m+m'-n-1 = m+m'-1.
\]
Hence, Theorem \ref{th : main 0-C} implies that the continuous map $f\colon M\longrightarrow N$, that was chosen arbitrary, admits a $k$-multiplicity.

\subsection{Proof of Corollary \ref{cor : main 2}}
Let $a\geq 1$ and $\ell\geq1$ be integers, let $k\geq 2$ be a power of $2$, and let $k(a+1)\leq 2^{\ell}-1$.
Furthermore, set $m=2^{\ell}-2-a$ and $n=2^{\ell}-2$.

It order to apply Theorem~\ref{th : main 1} we need first to find the integer
\[
r=m-s=\min\{\ell : \bar{w}_{m-\ell}(\RP^m)\neq 0\}.
\]
It is well known that the total Stiefel--Whitney class $w(\RP^m)=(1+t)^{m+1}$, where $H^*(\RP^m;\F_2)=\F_2[t]/\langle t^{m+1}\rangle$ and $\deg(t)=1$, consult for example \cite[Thm.\,4.5]{Milnor1974}.
Then 
\[
\bar{w}(\RP^m)=(1+t)^{2^{\ell}-m-1}=(1+t)^{a+1} = \sum_{i=0}^{a+1}{a+1 \choose i}t^i=1+(a+1)t+\cdots+t^{a+1}.
\]
Since $a+1\leq\tfrac1{k}(2^{\ell}-1)$ and $k\geq 2$ we have that $a+1\leq 2^{\ell-1}-1$.
Consequently, $m=2^{\ell}-2-a\geq 2^{l-1}$, and so $a+1\leq m$.
Thus, $r=m-a-1$.
Furthermore, $\bar{w}_{a+1}(\RP^m)^{k-1}=t^{(a+1)(k-1)}\neq 0$ because 
\[
(a+1)(k-1)=k(a+1)-a-1\leq 2^{\ell}-1-a-1=m.
\]
It remains to confirm that $r\leq 2m-n-1$.
Indeed, the following inequality holds:
\[
m-a-1=r\leq 2m-n-1 = m + 2^{\ell}-2-a -(2^{\ell}-2) - 1 = m-a-1
\]

Now, Theorem~\ref{th : main 1} applied to the manifolds $M=\RP^m$ and $N=\R^n$ concludes the proof of the corollary, meaning that every continuous map $\RP^{2^{\ell}-2-a}\longrightarrow \R^{2^{\ell}-2}$ admits a local $k$-multiplicity.

\subsection{Proof of Corollary \ref{cor : main 2.5}}
Let $a\geq 1$ and $\ell\geq1$ be integers, let $k\geq 2$ be a power of $2$, and let $k(a+1)\leq 2^{\ell}-1$.
Furthermore, set $m=2^{\ell}-2-a$ and $n=2^{\ell}-2$, and fix a continuous map $f\colon\RP^m\longrightarrow S^n$.

It order to apply Theorem~\ref{th : main 0} we need to find an integer $r=2m-1-s\leq 2m-1$ such that
	\[
	u_r(f^*\tau S^n\oplus(-\tau \RP^m)):=\det (w_{m+n-r-i+j})_{1\leq i,j\leq k-1}
	\]
does not vanish. 
Here $w_i:=w_i(f^*\tau S^n\oplus(-\tau \RP^m))$ is the $i$-th Stiefel--Whitney class of the vector bundle $f^*\tau S^n\oplus(-\tau \RP^m)$ for $i\geq 0$, and $w_i=0$ for $i<0$.
Since $w(\tau S^n)=1$ we have that $w(f^*\tau S^n)=1$.
Consequently, 
\[
w(f^*\tau S^n\oplus(-\tau \RP^m))=w(f^*\tau S^n)w(-\tau \RP^m)=w(-\tau \RP^m)=\bar{w}(\RP^m).
\]
As we have seen in the proof of Corollary \ref{cor : main 2} the dual Stiefel--Whitney class of $\RP^m$ is 
\[
\bar{w}(\RP^m)=(1+t)^{2^{\ell}-m-1}=(1+t)^{a+1} = \sum_{i=0}^{a+1}{a+1 \choose i}t^i=1+(a+1)t+\cdots+t^{a+1}.
\]
Again, $a+1\leq\tfrac1{k}(2^{\ell}-1)$ and $k\geq 2$ imply that  $a+1\leq 2^{\ell-1}-1$.
Hence, $m=2^{\ell}-2-a\geq 2^{l-1}$, and so $a+1\leq m$.
Thus, $\bar{w}_{a+1}(\RP^m)=t^{a+1}\neq 0$, and  $\bar{w}_i(\RP^m)=0$ for $i>a+1$.

Now, if we take $r=2^{\ell+1}-2a-5$ then $r=2m-1$ and
\begin{multline*}
	u_r(f^*\tau S^n\oplus(-\tau \RP^m)):=\det (w_{a+1-i+j})_{1\leq i,j\leq k-1}\qquad\qquad \\
	=\det (\bar{w}_{a+1-i+j}(\RP^m))_{1\leq i,j\leq k-1} = \bar{w}_{a+1}(\RP^m)^{k-1}=t^{(a+1)(k-1)}.
\end{multline*}
Since
$
(a+1)(k-1)=k(a+1)-a-1\leq 2^{\ell}-1-a-1=m
$
we have that $t^{(a+1)(k-1)}\neq 0$, and consequently $u_r(f^*\tau S^n\oplus(-\tau \RP^m))\neq 0$.

Thus, Theorem~\ref{th : main 0} applied to the manifolds $M=\RP^m$ and $N=S^n$ concludes the proof of the corollary.
\subsection{Proof of Corollary \ref{cor : main 3}}
Let $a\geq 1$ and $\ell\geq1$ be integers, let $k\geq 2$ be a power of $2$, and let $k(a-1)\leq 2^{\ell}-1$.
Now set $m=2^{\ell}-a$ and $n=2^{\ell+1}-3$.

Again we apply Theorem~\ref{th : main 1}.
First, we weed to find the integer $r=m-s$.
A know fact is that the total Stiefel--Whitney class $w(\CP^m)=(1+x)^{m+1}$, where $H^*(\CP^m;\F_2)=\F_2[x]/\langle x^{m+1}\rangle$ and $\deg(x)=2$, consult \cite[Thm.\,14.10]{Milnor1974}.
Therefore, the total dual Stiefel--Whitney class can be computed as follows 
\[
\bar{w}(\CP^m)=(1+x)^{2^{\ell}-m-1}=(1+x)^{a-1} = \sum_{i=0}^{a-1}{a-1 \choose i}x^i=1+(a-1)x+\cdots+x^{a-1}.
\]
From the assumption $k(a-1)\leq 2^{\ell}-1 \Leftrightarrow (k-1)(c-m-1)\leq m$ we derive that $a-1=2^{\ell}-m-1\leq \tfrac{m}{k-1}\leq m$. 
Thus, $2m-r=2a-2$ and so $r=2m-2a+2$

Next, $\bar{w}_{2a-2}(\CP^m)^{k-1}=x^{(a-1)(k-1)}\neq 0$ because $(k-1)(a-1)\leq m$. 
Finally, we verify that $r\leq 4m-n-1$.
Indeed,
\[
r=2m-2a+2=2m-2(2^{\ell}-m)+2=4m-2^{\ell+1}+2=4m-n-1.
\]

Again, Theorem~\ref{th : main 1} applied to the manifolds $M=\CP^m$ and $N=\R^n$ yields the statement of the corollary, that is 
every continuous map $\CP^{2^{\ell}-a}\longrightarrow \R^{2^{\ell+1}-3}$ admits a local $k$-multiplicity.

\subsection{Proof of Corollary \ref{cor : main 4}}

Let $a\geq 1$ and $\ell\geq1$ be integers, let $k$ be an odd prime, and let $2\leq a\leq\tfrac{k^{\ell}+1}{2}$.
Furthermore, let $m:=k^{\ell}-a$, $m'=\dim (-\tau \CP^m)$, and $n:=k^{\ell}-2$.
We will apply Theorem \ref{th : main 1-C} and therefore we verify that its assumptions are satisfied.

The total Chern class of the projective space is $c(\CP^m)=(1+x)^{m+1}$ where $H^*(\CP^m;\F_k)=\F_k[x]/\langle x^{m+1}\rangle$ and $\deg(x)=2$, see \cite[Thm.\,14.10]{Milnor1974}.
Since $2\leq a\leq\tfrac{k^{\ell}+1}{2}$ we have that
\[
\tfrac12(k^{\ell}-1)< m+1 < k^l 
\qquad\Longrightarrow\qquad
0 < k^{\ell}-m-1 < \tfrac{k^{\ell}+1}{2}\leq m+1.
\]
Therefore, we have 
\[
\bar{c}(\CP^m)=(1+x)^{k^{\ell}-m-1}=(1+x)^{a-1} = \sum_{i=0}^{a-1}{a-1 \choose i}x^i=1+(a-1)x+\cdots+x^{a-1},
\]
where binomial coefficients are considered mod $k$.
Thus $\bar{c}_{a-1}(\CP^m)\neq 0$ and $\bar{c}_{i}(\CP^m)=0$ for all $i>a-1$.
Consequently $m'\geq a-1$ and
\[
r=m'-s= \min\{\ell : \bar{c}_{m'-\ell}(\CP^m)\neq 0\} = m'-a+1.
\]
Since $k(a-1)\leq k^{\ell}-1\ \Longleftrightarrow \ (k-1)(a-1)\leq k^{\ell}-a$ we have that $\bar{c}_{a-1}(\CP^m)^{k-1}\neq 0$

It remains to verify that $r\leq m+m'-n-1$.
Indeed,
\[
r = m'-a+1 = m' + m - k^{\ell} +1 = m' + m - n - 2 +1 = m' + m - n - 1.
\]
Hence, Theorem~\ref{th : main 1-C} applied to the manifolds $M=\CP^m$ and $N=\C^n$ yields the statement of the corollary, that is 
every continuous map $\CP^{k^{\ell}-a}\longrightarrow \C^{k^{\ell}-2}$ admits a local $k$-multiplicity.

\small
\providecommand{\noopsort}[1]{}



\begin{thebibliography}{10}
\itemsep=0pt


\bibitem{Barany2016}
Imre B{\'a}r{\'a}ny, Daniel Hug, and Rolf Schneider, \emph{Affine diameters of
  convex bodies}, Proc. Amer. Math. Soc. \textbf{144} (2016), no.~2, 797--812.

\bibitem{Barany1990}
Imre B{\'a}r{\'a}ny and Tudor Zamfirescu, \emph{Diameters in typical convex
  bodies}, Canad. J. Math. \textbf{42} (1990), no.~1, 50--61.

\bibitem{Blagojevic2015-02}
Pavle V.~M. Blagojevi\'c, Frederick~R. Cohen, Wolfgang L\"uck, and G\"unter~M.
  Ziegler, \emph{On complex highly regular embeddings and the extended
  {V}assiliev conjecture}, Int. Math. Research Notes (IMRN), Published online
  December 17, 2015, \url{doi:10.1093/imrn/rnv341};
  \href{hhttp://arxiv.org/abs/1410.6052}{arXiv:1410.6052}.

\bibitem{Blagojevic2105}
Pavle V.~M. Blagojevi{\'c}, Wolfgang L{\"u}ck, and G{\"u}nter~M. Ziegler,
  \emph{Equivariant topology of configuration spaces}, J. Topol. \textbf{8}
  (2015), no.~2, 414--456.

\bibitem{Blagojevic2016}
\bysame, \emph{On highly regular embeddings}, Trans. Amer. Math. Soc.
  \textbf{368} (2016), no.~4, 2891--2912.

\bibitem{Bott1982}
Raoul Bott and Loring~W. Tu, \emph{Differential forms in algebraic topology},
  Graduate Texts in Mathematics, vol.~82, Springer-Verlag, New York-Berlin,
  1982.

\bibitem{Bredon1993}
Glen~E. Bredon, \emph{Topology and geometry}, Graduate Texts in Mathematics,
  vol. 139, Springer-Verlag, New York, 1993.

\bibitem{Brown1977}
Edgar~H.jun. {Brown} and Franklin~P. {Peterson}, \emph{{On immersions of
  n-manifolds}}, {Adv. Math.} \textbf{24} (1977), 74--77.

\bibitem{Cohen1976}
Frederick~R. Cohen, \emph{{The homology of $C_{n+1}$-spaces, $n\geq 0$}}, ``The
  Homology of Iterated Loop Spaces'', Lecture Notes in Math. \textrm{533},
  Springer, 1976, pp.~207--351.

\bibitem{CohenR1985}
Ralph~L. {Cohen}, \emph{{The immersion conjecture for differentiable
  manifolds}}, {Ann. Math. (2)} \textbf{122} (1985), 237--328.

\bibitem{Green1978}
R.~E. Greene, \emph{Complete metrics of bounded curvature on noncompact
  manifolds}, Arch. Math. (Basel) \textbf{31} (1978/79), no.~1, 89--95.

\bibitem{Gromov2010}
Mikhail Gromov, \emph{Singularities, expanders and topology of maps. {P}art 2:
  {F}rom combinatorics to topology via algebraic isoperimetry}, Geom. Funct.
  Anal. \textbf{20} (2010), no.~2, 416--526.

\bibitem{Grunbaum1963}
Branko Gr{\"u}nbaum, \emph{Measures of symmetry for convex sets}, Proc.
  {S}ympos. {P}ure {M}ath., {V}ol. {VII}, Amer. Math. Soc., Providence, R.I.,
  1963, pp.~233--270.

\bibitem{Macdonald1992}
I.~G. Macdonald, \emph{Schur functions: theme and variations}, S\'eminaire
  {L}otharingien de {C}ombinatoire ({S}aint-{N}abor, 1992), Publ. Inst. Rech.
  Math. Av., vol. 498, Univ. Louis Pasteur, Strasbourg, 1992, pp.~5--39.

\bibitem{Macdonald1995-Book}
\bysame, \emph{Symmetric functions and {H}all polynomials}, second ed., Oxford
  Mathematical Monographs, The Clarendon Press, Oxford University Press, New
  York, 1995, With contributions by A. Zelevinsky, Oxford Science Publications.

\bibitem{Milnor1957}
John~W. Milnor, \emph{Lectures on characteristic classes}, Published in 197 as
  Annals of Math. Studies 76, with Stasheff, 1957.

\bibitem{Milnor1974}
John~W. Milnor and James~D. Stasheff, \emph{Characteristic classes}, Princeton
  University Press, Princeton, 1974, Annals of Mathematics Studies, No. 76.

\bibitem{Soltan2005}
V.~Soltan, \emph{Affine diameters of convex-bodies---a survey}, Expo. Math.
  \textbf{23} (2005), no.~1, 47--63.

\bibitem{Stong1968}
Robert~E. Stong, \emph{Notes on cobordism theory}, Mathematical notes,
  Princeton University Press, Princeton, N.J.; University of Tokyo Press,
  Tokyo, 1968.

\bibitem{Thomas1959}
Emery Thomas, \emph{On tensor products of {$n$}-plane bundles}, Arch. Math.
  \textbf{10} (1959), 174--179.

\end{thebibliography}
\end{document}